\documentclass[a4paper, 12pt]{amsart}
\usepackage{amsmath, amssymb, amsthm, bm}	
\usepackage{braket, stmaryrd,mathrsfs}
\usepackage{cite}
\usepackage[all]{xy}
\usepackage{youngtab}					
\usepackage[dvipdfmx]{graphicx,xcolor}
\usepackage{tikz}

\calclayout

\newcommand{\pr}[1]{#1^{\prime}}

\newcommand{\del}{\partial}

\newcommand{\mfrak}[1]{\mathfrak{#1}}
\newcommand{\mcal}[1]{\mathcal{#1}}
\newcommand{\mbb}[1]{\mathbb{#1}}
\newcommand{\mrm}[1]{\mathrm{#1}}

\makeatletter
\newcommand{\xRightarrow}[2][]{\ext@arrow 0359\Rightarrowfill@{#1}{#2}}
\makeatother

\theoremstyle{plain}
\newtheorem{thm}{Theorem}[section]
\newtheorem*{thm*}{Theorem}
\newtheorem{lem}[thm]{Lemma}
\newtheorem{prop}[thm]{Proposition}

\theoremstyle{definition}
\newtheorem{defn}[thm]{Definition}
\newtheorem{prob}{Problem}
\theoremstyle{remark}
\newtheorem{rem}[thm]{Remark}
\newtheorem{exam}[thm]{Example}

\makeatletter
    
    \@addtoreset{equation}{section}
\makeatother

\makeatletter
	
	\@addtoreset{figure}{section}
\makeatother

\begin{document}

\title[Conformal welding problem]{Conformal welding problem, flow line problem, and multiple Schramm--Loewner evolution}
\author{Makoto Katori}
\address[Makoto KATORI]{Department of Physics, Faculty of Science and Engineering, Chuo University, Kasuga, Bunkyo-ku, Tokyo 112-8551, Japan}
\email{katori@phys.chuo-u.ac.jp}

\author{Shinji Koshida}
\address[Shinji KOSHIDA]{Department of Basic Science, The University of Tokyo, Komaba, Meguro, Tokyo 153-8902, Japan}
\email{koshida@vortex.c.u-tokyo.ac.jp}
\email{koshida@phys.chuo-u.ac.jp}

\begin{abstract}
A quantum surface (QS) is an equivalence class of pairs $(D,H)$ of simply connected domains $D\subsetneq\mbb{C}$
and random distributions $H$ on $D$ induced by the conformal equivalence for random metric spaces.
This distribution-valued random field is extended to a QS with $N+1$ marked boundary points (MBPs) with $N\in\mbb{Z}_{\ge 0}$.
We propose the conformal welding problem for it in the case of $N\in\mbb{Z}_{\ge 1}$.
If $N=1$, it is reduced to the problem introduced by Sheffield, who solved it by coupling the QS with the Schramm--Loewner evolution (SLE).
When $N \ge 3$, there naturally appears room of making the configuration of MBPs random,
and hence a new problem arises how to determine the probability law of the configuration.
We report that the multiple SLE in $\mathbb{H}$ driven by the Dyson model on $\mathbb{R}$ 
helps us to fix the problems and makes them solvable for any $N \ge 3$.
We also propose the flow line problem for an imaginary surface with boundary condition changing points (BCCPs).
In the case when the number of BCCPs is two, this problem was solved by Miller and Sheffield.
We address the general case with an arbitrary number of BCCPs in a similar manner to the conformal welding problem.
We again find that the multiple SLE driven by the Dyson model plays a key role
to solve the flow line problem.
\end{abstract}


\maketitle


\section{Introduction}
\label{sect:intro}

{\it Gaussian free field} (GFF) \cite{Sheffield2007} in two dimensions gives a mathematically rigorous formulation 
of the free bose field, a model of two-dimensional conformal field theory (CFT) \cite{BelavinPolyakovZamolodchikov1984}.
It relies on the probability theory and, conceptually, realizes the path-integral over field configurations with weight defined through the action functional
\begin{equation*}
	S[h]=\int |\nabla h|^{2},
\end{equation*}
where $h$ is a real-valued field configuration on a two-dimensional domain.
A detailed formulation of GFF as CFT has been established in the booklet~\cite{KangMakarov2013}.

Besides a formulation of CFT, GFF also turns out to give a rich playground for random geometry.
In fact, an instance of GFF is regarded as a distribution on a domain, from which one extracts geometric data, typically, in two manners.
In the first one, we define a random metric on the domain by exponentiating the GFF.
Such defined random metric formulates the {\it Liouville quantum gravity} (LQG) \cite{DuplantierSheffield2011}, a model of two-dimensional quantum gravity,
the original idea of which was given by Polyakov \cite{Polyakov1981a,Polyakov1981b}.
In the other one, we consider a flow line of the vector field $e^{ih/\chi}$,
where we identify the two dimensional domain with one in the complex plane.
Compared to the LQG, this type of geometry is often called the {\it imaginary geometry} \cite{MillerSheffield2016a,MillerSheffield2016b,MillerSheffield2016c,MillerSheffield2017}.
Owing to the recent studies \cite{Dubedat2009,MillerSheffield2016a,Sheffield2016}, it has been clarified that both of these models of random geometry
are closely related to the theory of {\it Schramm--Loewner evolution} (SLE).

SLE was first introduced in ~\cite{Schramm2000} as a candidate for a stochastic model that describes a cluster interface
in a two-dimensional critical lattice model at a scaling limit, relying on the stochastic analogue of the Loewner theory \cite{Loewner1923, KufarevSobolevSporyseva1968}.
After its introduction, many authors studied SLE extensively revealing its properties \cite{RohdeSchramm2005} and the relation to lattice models \cite{Smirnov2001,ChelkakDuminil-CopinHonglerKemppainenSmirnov2014}.
Useful expositions of SLE can be found in ~\cite{Werner2004a,Lawler2009a}.

The relation between SLE and the LQG or the imaginary geometry is formulated by considering couplings of SLEs and GFFs.
Roughly speaking, a GFF perturbed by an appropriate harmonic function is shown to be stationary under the operation of cutting the domain along an SLE path.
Based on this fact, it can be argued \cite{Dubedat2009,MillerSheffield2016a,Sheffield2016} that an instance of GFF determines an SLE path compatibly to the LQG or the imaginary geometry.

Then it seems natural to ask how the coupling between an SLE and a GFF can be extended to the case of multiple SLE and how it can be interpreted in the context of the LQG or the imaginary geometry. In the present paper, we address this issue by considering the conformal welding problem for a quantum surface with multiple marked boundary points and the flow line problem for an imaginary surface with multiple boundary condition changing points. Consequently, we will find that the multiple SLE driven by a time change of the Dyson model plays an analogous role as the one that an SLE plays in the original works \cite{Sheffield2016,MillerSheffield2016a}.

The present paper is organized as follows:
In the following Sect.\ \ref{sect:preliminaries}, we make preliminaries on GFF and multiple SLE that are needed in the remaining part of the paper.
In Sect.\ \ref{sect:formulation_problem}, with a short review on the classical matter on the LQG and the imaginary geometry, we formulate our problems, the conformal welding problem and the flow line problem. We also present the main results Theorems \ref{thm:answer_welding} and \ref{thm:answer_flowline} there, which are proved in the succeeding Sects.\ \ref{sect:coupling_reverse} and \ref{sect:coupling_forward}, respectively.
In the final Sect.\ \ref{sect:concluding_remarks}, we discuss related topics and future directions.
In Appendix \ref{app:construction}, we give a construction of the spaces $\mcal{S}_{\gamma}$ and $\mcal{S}_{\gamma,N+1}$ of pre-quantum surfaces and ones with marked boundary points as orbifolds and investigate their structures in detail.
In Appendix \ref{app:driving_processes}, we summarize the approach in ~\cite{BauerBernardKytola2005} that formulated multiple SLE in relation to CFT and the analogous way of defining the reverse flow of multiple SLE.

\subsection*{Acknowledgments}
The authors would like to thank Takuya Murayama, Ikkei Hotta and Daiya Yamauchi for useful discussion.
The authors are also grateful to Sebastian Schleissinger for his valuable comments on the first version of the manuscript and Kalle Kyt{\"o}l{\"a} for discussions on multiple SLEs.
The present work was partially supported by the Research Institute for Mathematical Sciences (RIMS),
a Joint Usage/Research Center located in Kyoto University.
The authors thank Naotaka Kajino, Takashi Kumagai, and Daisuke Shiraishi for organizing the very fruitful workshop,
``RIMS Research Project: Gaussian Free Fields and Related Topics", held in 18-21 September 2018 at RIMS.
MK was supported by
the Grant-in-Aid for Scientific Research (C) (No.26400405),
(C) (No.19K03674),
(B) (No.18H01124), and
(S) (No.16H06338)
of Japan Society for the Promotion of Science (JSPS).
SK was supported by the Grant-in-Aid for JSPS Fellows (No. 17J09658, No. 19J01279).

\section{Preliminaries}
\label{sect:preliminaries}
This section is devoted to preliminaries on GFF and multiple SLE. In Subsect.\ \ref{subsect:GFF}, we review the notion of GFF both under Dirichlet boundary condition and free boundary condition, and in Subsect.\ \ref{subsect:multiple_SLE}, we fix a framework of multiple SLE starting from a Loewner chain for a family of multiple slit domains.
\subsection{Gaussian free field}
\label{subsect:GFF}
\subsubsection*{\bf Dirichlet boundary case}
Let $D \subsetneq\mbb{C}$ be a simply connected domain
and let $W^{\mrm{Dir}}(D)$ be the Hilbert space completion of the space $C_{0}^{\infty}(D)$ of smooth functions supported in $D$ with respect to the Dirichlet inner product
\begin{equation}
	\label{eq:Dirichlet_innter_product}
	(f,g)_{\nabla}=\frac{1}{2\pi}\int_{D} (\nabla f)(z)\cdot (\nabla g)(z)d\mu(z),\ \ f, g\in C^{\infty}_{0}(D),
\end{equation}
where $\mu$ is the Lebesgue measure on $D\subset\mbb{C}$; $d\mu(z)=\frac{\sqrt{-1}}{2}dzd\overline{z}$.
A GFF with Dirichlet boundary condition on $D$ is defined as an isometry 
$H^{\mrm{Dir}}_{D}:W^{\mrm{Dir}}(D)\to L^{2}(\Omega^{\mrm{Dir}}_{D},\mcal{F}^{\mrm{Dir}}_{D},\mbb{P}^{\mrm{Dir}}_{D})$
where $(\Omega^{\mrm{Dir}}_{D},\mcal{F}^{\mrm{Dir}}_{D},\mbb{P}^{\mrm{Dir}}_{D})$ is a probability space
such that each $(H^{\mrm{Dir}}_{D},\rho)_{\nabla}:=H^{\mrm{Dir}}_{D}(\rho)$, $\rho\in W^{\mrm{Dir}}(D)$, is a mean-zero Gaussian random variable \cite{Sheffield2007}.
One can construct such an isometry relying on the Bochner-Minlos theorem
that is an analogue of Bochner's theorem applicable to the case when the source Hilbert space is infinite dimensional ~\cite[Chapter 3]{Hida1980}.
It is also known that, in this construction, the sigma field $\mcal{F}^{\mrm{Dir}}_{D}$ is generated by
the image of $W^{\mrm{Dir}}(D)$ under $H^{\mrm{Dir}}_{D}$, i.e., $H^{\mrm{Dir}}_{D}$ is {\it full}.

Let us denote by $C^{\infty}_{0}(D)^{\prime}$ the space of distributions with test functions in $C^{\infty}_{0}(D)$ with respect to the topology of $W^{\mrm{Dir}}(D)$.
From a construction of a Dirichlet boundary GFF, for each $\omega\in \Omega^{\mrm{Dir}}_{D}$, the assignment
\begin{equation*}
	H^{\mrm{Dir}}_{D}(\omega):C^{\infty}_{0}(D)\to \mbb{C},\ \ \rho\mapsto (H^{\mrm{Dir}}_{D},\rho)_{\nabla}(\omega)
\end{equation*}
is continuous so that $H^{\mrm{Dir}}_{D}(\omega)\in C^{\infty}_{0}(D)^{\prime}$.
Therefore a GFF under Dirichlet boundary condition can be regarded as a random distribution
$H^{\mrm{Dir}}_{D}:\Omega^{\mrm{Dir}}_{D}\to C^{\infty}_{0}(D)^{\prime}$.

For $\rho\in C_{0}^{\infty}(D)$, we define $(H^{\mrm{Dir}}_{D},\rho):=2\pi(H^{\mrm{Dir}}_{D},(-\Delta)^{-1}\rho)_{\nabla}$.
Then we have
\begin{equation*}
	\mbb{E}[(H^{\mrm{Dir}}_{D},\rho_{1})(H^{\mrm{Dir}}_{D},\rho_{2})]=\int_{D\times D}\rho_{1}(z)G^{\mrm{Dir}}_{D}(z,w)\rho_{2}(w)d(\mu\otimes\mu)(z,w),
\end{equation*}
$\rho_{1},\rho_{2}\in C^{\infty}_{0}(D)$, where $G^{\mrm{Dir}}_{D}(z,w)$ is the Green function on $D$ with Dirichlet boundary condition.
For example, when $D$ is the upper half plane $\mbb{H}$,
$G^{\mrm{Dir}}_{\mbb{H}}(z,w)=-\log|z-w|+\log|z-\overline{w}|$.
When we regard the GFF as a random distribution on $D$,
it is reasonable to symbolically express $(H^{\mrm{Dir}}_{D},\rho)$ for $\rho\in C_{0}^{\infty}(D)$ as
\begin{equation*}
	(H^{\mrm{Dir}}_{D},\rho)=\int_{D}H^{\mrm{Dir}}_{D}(z)\rho(z)d\mu(z).
\end{equation*}
We understand the object $H^{\mrm{Dir}}_{D}(\cdot)=H^{\mrm{Dir}}_{D}(\cdot,\omega)$, $\omega\in \Omega^{\mrm{Dir}}_{D}$, in this sense.

\subsubsection*{\bf Free boundary case}
A GFF with free boundary condition is defined in a similar way, but on a different space of test functions.
Let $C_{\nabla}^{\infty}(D)$ be the space of smooth functions on $D$ whose gradients are compactly supported in $D$ and whose total-masses are zero: $\int f d\mu=0$, $f\in C_{\nabla}^{\infty}(D)$.
Then we denote the Hilbert space completion of $C_{\nabla}^{\infty}(D)$ with respect to the Dirichlet inner product defined by the same formula as (\ref{eq:Dirichlet_innter_product}) by $W^{\mrm{Fr}}(D)$.
A GFF with free boundary condition is defined as an isometry
$H^{\mrm{Fr}}_{D}:W^{\mrm{Fr}}(D)\to L^{2}(\Omega^{\mrm{Fr}}_{D},\mcal{F}^{\mrm{Fr}}_{D},\mbb{P}^{\mrm{Fr}}_{D})$,
where $(\Omega^{\mrm{Fr}}_{D},\mcal{F}^{\mrm{Fr}}_{D},\mbb{P}^{\mrm{Fr}}_{D})$ is a probability space.
This can also be regarded as a random distribution $H^{\mrm{Fr}}_{D}:\Omega^{\mrm{Fr}}_{D}\to C_{\nabla}^{\infty}(D)^{\prime}$,
where $C_{\nabla}^{\infty}(D)^{\prime}$ denote the space of distributions with test functions in $C_{\nabla}^{\infty}(D)$.
A significant difference of the free boundary GFF from the Dirichlet boundary one lies in the fact that the GFF with free boundary condition $H^{\mrm{Fr}}_{D}$ is regarded as a random distribution on $D$ modulo additive constants.
Since, in this paper, we only treat simply connected domains, it suffices to present
the GFF with free boundary condition on $\mbb{H}$ (see ~\cite{Sheffield2016,Berestycki2016,QianWerner2018}).
Let $\rho_{1}, \rho_{2}\in C_{\nabla}^{\infty}(\mbb{H})$ be test functions.
Then $(H^{\mrm{Fr}}_{\mbb{H}},\rho_{i})$, $i=1,2$ are mean-zero Gaussian variables with covariance
\begin{equation*}
	\mbb{E}[(H^{\mrm{Fr}}_{\mbb{H}},\rho_{1})(H^{\mrm{Fr}}_{\mbb{H}},\rho_{2})]=\int_{\mbb{H}\times\mbb{H}}\rho_{1}(z)G^{\mrm{Fr}}_{\mbb{H}}(z,w)\rho_{2}(w)d(\mu\otimes\mu)(z,w),
\end{equation*}
where $G^{\mrm{Fr}}_{\mbb{H}}(z,w)=-\log |z-w|-\log |z-\overline{w}|$.

\subsection{Multiple SLE}
\label{subsect:multiple_SLE}
We begin with recalling the result in ~\cite{RothSchleissinger2017} that dealt with the multiple-slit version of the Loewner theory \cite{Loewner1923,KufarevSobolevSporyseva1968}.
Let $\eta^{(i)}:(0,\infty) \to\mbb{H}$, $i=1,\dots, N$, be non-intersectiong curves in $\mbb{H}$
anchored on $\mbb{R}$: $\eta^{(i)}_{0}:=\lim_{t\to 0}\eta^{(i)}(t)\in \mbb{R}$, $i=1,\dots, N$.
We set $\mbb{H}^{\eta}_{t}:=\mbb{H}\backslash\left(\bigcup_{i=1}^{N}\eta^{(i)}(0,t]\right)$.
Then, at each time $t\in (0,\infty)$, there is a unique conformal mapping $g_{t}:\mbb{H}^{\eta}_{t}\to\mbb{H}$ under the hydrodynamic normalization (Fig.~\ref{fig:SLE_slits}):
\begin{equation*}
	g_{t}(z)=z+\frac{C(\mbb{H}^{\eta}_{t})}{z}+O(|z|^{-2}),\ \ z\to\infty.
\end{equation*}
The constant $C(\mbb{H}^{\eta}_{t})$ is called the half plane capacity of $\mbb{H}^{\eta}_{t}$.
Notice that, by changing the parametrization of the curves, we can take $C(\mbb{H}^{\eta}_{t})=2Nt$.
We call such a parametrization of curves a standard parametrization.

\begin{thm}[Alternative expression of {\cite[Theorem 1.1]{RothSchleissinger2017}}]
\label{thm:multple_Loewner}
Let $\eta^{(i)}:(0,\infty)\to\mbb{H}$, $i=1,\dots, N$, be non-intersectiong curves in $\mbb{H}$
anchored on $\mbb{R}$ with a standard parametrization.
There exists a unique set of continuous driving functions $\bm{X}_{t}=(X_{t}^{(1)},\dots, X_{t}^{(N)})\in \mbb{R}^{N}$, $t\in [0,\infty)$, such that 
the family of conformal mappings $g_{t}:\mbb{H}^{\eta}_{t}\to \mbb{H}$ solves the multiple Loewner equation:
\begin{equation}
\label{eq:multiple_Loewner}
	\frac{d}{dt}g_{t}(z)=\sum_{i=1}^{N}\frac{2}{g_{t}(z)-X^{(i)}_{t}},\ \ t\ge 0,\ \ \ g_{0}(z)=z\in\mbb{H},
\end{equation}
i.e., $\{g_{t}\}_{t\ge 0}$ is the Loewner chain driven by $\{\bm{X}_{t}:t\ge 0\}$.
Moreover, the driving functions are determined by $X^{(i)}_{t}=\lim_{z\to \eta^{(i)}(t), z\in\mbb{H}^{\eta}_{t}}g_{t}(z)$, $i=1,\dots, N$. 
\end{thm}

\begin{figure}[h]
\centering
\includegraphics[width=\hsize]{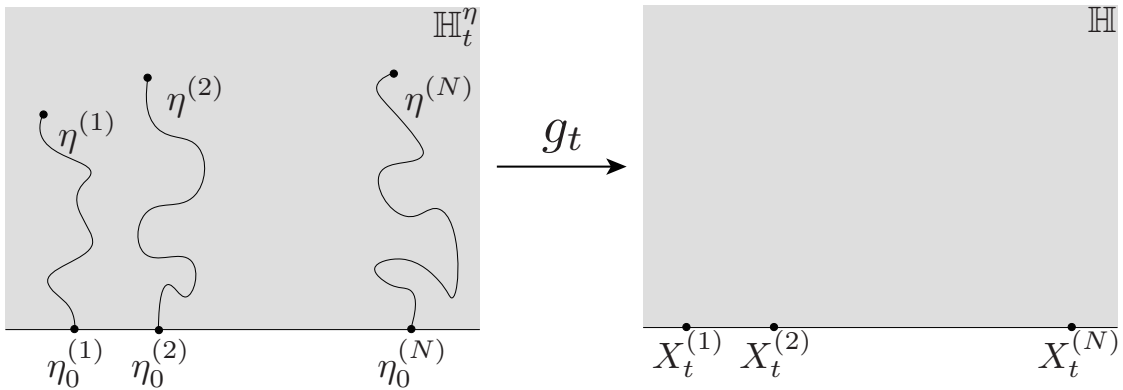}
\caption{Uniformization of the domain $\mbb{H}^{\eta}_{t}$ by the Loewner chain.}
\label{fig:SLE_slits}
\end{figure}

By virtue of the above theorem, an $N$-tuple of random slits in $\mbb{H}$ anchored on $\mbb{R}$
is converted to a stochastic process $\{\bm{X}_{t}\in\mbb{R}^{N}:t\ge 0\}$.
We assume that it solves a system of stochastic differential equations (SDEs)
\begin{equation}
	\label{eq:driving_process_Schramm}
	dX^{(i)}_{t}=\sqrt{\kappa}dB_{t}^{(i)}+F^{(i)}(\bm{X}_{t})dt,\ \  i=1,\dots, N,\ \ t\ge 0,
\end{equation}
where $\{ B_{t}^{(i)}:t\ge 0\}_{i=1}^{N}$ are mutually independent one-dimensional standard Brownian motions and $\{F^{(i)}(\bm{x})\}_{i=1}^{N}$
are suitable functions of $\bm{x}=(x_{1},\dots, x_{N})$ not explicitly dependent on $t$,
i.e., $\{\bm{X}_{t}:t\ge 0\}$ are of the Markovian type (see \cite[Eq. (2.11) in Chapter IV]{IkedaWatanabe1989}).
Then the solution $\{g_{t}\}_{t\ge 0}$ is just the multiple SLE considered in ~\cite{BauerBernardKytola2005,Graham2007}.
While, in ~\cite{BauerBernardKytola2005}, the set of driving processes was derived from a single {\it auxiliary function} in relation to CFT
as is summarized in Appendix \ref{app:driving_processes},
here we do not assume any CFT origin of the multiple SLE.
We write SLE$_{\kappa}$ if we need to specify the parameter $\kappa$.

\section{Formulation of problems and results}
\label{sect:formulation_problem}
In this section, we formulate the problems we address in this paper, the conformal welding problem and the flow line problem, and describe the results.
These two problems are, in fact, very similar in the sense that the only difference is the boundary condition of a GFF and solved analogously by means of multiple SLE, but have different origins in the random geometry. Thus we present them separately.
\subsection{Conformal welding problem}
\label{subsect:conformal_welding_problem}
\subsubsection*{\bf Liouville quantum gravity}
GFF plays a relevant role in constructing the LQG \cite{DuplantierSheffield2011} (see also ~\cite{RhodesVargas2017,DavidKupiainenRhodesVargas2016,HuangRhodesVargas2018}).
Following the original idea by Polyakov \cite{Polyakov1981a,Polyakov1981b}, it is expected that the object $e^{\gamma H^{\mrm{Fr}}_{D}(z)}d\mu(z)$
is the desired random area measure on $D$, where $\gamma\in (0,2)$.
This does not work, however, because each realization of $H^{\mrm{Fr}}_{D}$,
$h(\cdot)=H^{\mrm{Fr}}_{D}(\cdot,\omega)$, $\omega\in \Omega^{\mrm{Fr}}_{D}$, is not a function but a distribution on $D$,
thus its exponentiation has to be verified in some sense.
This difficulty is overcome by a certain regularization.
Let us fix a realization $h\in C_{\nabla}^{\infty}(D)^{\prime}$ and let $h_{\epsilon}(z)$ be the mean value of $h$ on the circle $\del B_{\epsilon}(z)$
of radius $\epsilon$ centered at $z \in D$ assuming that the distance from $z$ to the boundary is larger than $\epsilon$.
Then the required area measure is obtained by
\begin{equation*}
	d\mu^{\gamma}_{h}(z):=\lim_{\epsilon\to 0}\epsilon^{\gamma^{2}/2}e^{\gamma h_{\epsilon}(z)}d\mu(z),\ \ z\in D.
\end{equation*}
In a similar way, one can construct a linear measure on the boundary
\begin{equation*}
	d\nu^{\gamma}_{h}(x):=\lim_{\epsilon\to 0}\epsilon^{\gamma^{2}/4}e^{\gamma h_{\epsilon}(x)/2}d\nu(x),\ \ x\in \del D,
\end{equation*}
where $\nu$ is the Lebesgue measure on the boundary,
while, in this case, $h_{\epsilon}(x)$ is the average over the semi-circle centered at $x\in \del D$ of radius $\epsilon$ included in $D$.

Let $\tilde{D}\subsetneq\mbb{C}$ be another simply connected domain, and $\psi:\tilde{D}\to D$ be a conformal equivalence.
Then an area measure is induced on $\tilde{D}$ by pulling back the measure $\mu^{\gamma}_{h}$ on $D$.
That is, for a measurable set $A\subset \tilde{D}$, its area is computed as $\psi^{\ast}\mu^{\gamma}_{h}(A):=\mu_{h}^{\gamma}(\psi(A))$.
When we closely look at the pulled-back measure, we find that it can also be realized as $\mu^{\gamma}_{\tilde{h}}$
built from a distribution $\tilde{h}$ on $\tilde{D}$.
Indeed, by changing integration variables,
the area of $A\subset \tilde{D}$ with respect to the pulled-back measure $\psi^{\ast}\mu^{\gamma}_{h}$ becomes
\begin{equation*}
	\lim_{\epsilon\to 0}\int_{\psi(A)}\epsilon^{\gamma^{2}/2}e^{\gamma h_{\epsilon}(z)}d\mu(z)
	=\lim_{\epsilon \to 0}\int_{A}(|\pr{\psi}(w)|\epsilon)^{\gamma^{2}/2}e^{\gamma (h\circ \psi)_{\epsilon}(w)}|\pr{\psi}(w)|^{2}d\mu(w),
\end{equation*}
where $\pr{\psi}(w)=\frac{d\psi}{dw}(w)$.
Note that, in the right hand side, in which the integral is taken over $A\subset \tilde{D}$,
the regularization parameter $\epsilon$ has to be rescaled by $|\pr{\psi}(w)|$.
This implies that if we introduce a distribution on $\tilde{D}$ by $\tilde{h}=h\circ \psi +Q\log |\pr{\psi}|$
with the parametrization $Q=(\frac{\gamma^{2}}{2}+2)/\gamma=\frac{2}{\gamma}+\frac{\gamma}{2}$,
then the corresponding area measure $\mu^{\gamma}_{\tilde{h}}$ agrees with the pulled-back measure $\psi^{\ast}\mu^{\gamma}_{h}$.
It can be verified that the boundary measure also behaves correctly:
$\nu^{\gamma}_{\tilde{h}}(I)=\nu^{\gamma}_{h}(\psi(I))$ for a measurable $I\subset \del\tilde{D}$.
Motivated by this, we make the following definition:

\begin{defn}
Let $\gamma\in (0,2)$.
Pairs $(D_{1},h_{1})$ and $(D_{2},h_{2})$ of simply connected domains $D_{i}\subsetneq\mbb{C}$
and distributions $h_{i}\in C_{\nabla}^{\infty}(D_{i})^{\prime}$, $i=1,2$,
are said to be $\gamma$-equivalent if there exists a conformal equivalence  $\psi:D_{1}\to D_{2}$ such that
\begin{equation*}
	h_{1}=h_{2}\circ \psi + Q\log |\pr{\psi}|
\end{equation*}
holds, where $Q=\frac{2}{\gamma}+\frac{\gamma}{2}$.
\end{defn}

\begin{defn}[Pre-quantum surface]
Let $\gamma\in (0,2)$.
A {\it $\gamma$-pre-quantum surface} is an $\gamma$-equivalence class of pairs $(D,h)$ of simply connected domains $D\subsetneq\mbb{C}$
and distributions $h\in C_{\nabla}^{\infty}(D)^{\prime}$.
We denote the $\gamma$-equivalence class of $(D,h)$ by $[D,h]_{\gamma}$ and 
write the collection of $\gamma$-pre-quantum surfaces as
\begin{equation*}
	\mcal{S}_{\gamma}:=\left\{[D,h]_{\gamma}|D\subsetneq \mbb{C},h\in C_{\nabla}^{\infty}(D)^{\prime}\right\}.
\end{equation*}
\end{defn}

We will give the construction of $\mcal{S}_{\gamma}$ as an orbifold in Appendix \ref{app:construction}.

The quantization of $\gamma$-pre-quantum surfaces is carried out by randomizing them:
\begin{defn}[Quantum surface]
Let $\gamma\in (0,2)$.
A {\it $\gamma$-quantum surface} is a probability measure on $\mcal{S}_{\gamma}$.
Equivalently, a $\gamma$-quantum surface is a collection of pairs $(D,H_{D})$, where $D\subsetneq \mbb{C}$ is a simply connected domain
and $H_{D}$ is a $C_{\nabla}^{\infty}(D)^{\prime}$-valued random field subject to the condition
that, for all simply connected domains $D_{1}, D_{2}\subsetneq \mbb{C}$ and conformal equivalences $\psi:D_{1}\to D_{2}$,
the equality in probability law
\begin{equation}
\label{eq:quantum_equivalence}
	H_{D_{1}}\overset{\mrm{(law)}}{=}H_{D_{2}}\circ \psi+Q\log |\pr{\psi}|
\end{equation}
holds, where $Q=\frac{2}{\gamma}+\frac{\gamma}{2}$.
We write this collection as $[D,H_{D}]_{\gamma}$.
\end{defn}

\begin{rem}
If one has a pair $(D,H_{D})$ of a simply connected domain $D\subsetneq\mbb{C}$ and a $C_{\nabla}^{\infty}(D)^{\prime}$-valued
random field $H_{D}$, then it uniquely extends to a $\gamma$-quantum surface $[D,H_{D}]_{\gamma}$.
\end{rem}

\begin{exam}
A relevant example of a $\gamma$-quantum surface arises from the free boundary GFF.
A pair $(D,H_{D}^{\mrm{Fr}})$ of a simply connected domain $D\subsetneq\mbb{C}$
and the free boundary GFF $H_{D}^{\mrm{Fr}}$ on $D$ defines a quantum surface.
Indeed, the assignment $\omega\mapsto [D,H^{\mrm{Fr}}_{D}(\omega)]_{\gamma}$
gives an $\mcal{S}_{\gamma}$-valued random field on $\Omega^{\mrm{Fr}}_{D}$
and induces a probability measure on $\mcal{S}_{\gamma}$.
\end{exam}

\subsubsection*{\bf Quantum surface with marked boundary points}
In order to formulate and address the conformal welding problem, we define a {\it quantum surface with marked boundary points}, which has also been introduced by an earlier literature ~\cite{DuplantierMillerSheffield2014}.
It is a refined version of a quantum surface decorated by data of boundary points.
Let $D\subsetneq \mbb{C}$ be a simply connected domain.
For $N\in\mbb{Z}_{\ge 0}$, we define the configuration space for $N+1$ ordered boundary points of $D$ as
\begin{equation*}
	\mrm{Conf}^{<}_{N+1}(\del D)=\left\{(x_{1},\dots, x_{N+1})\in (\del D)^{N+1}|\substack{x_{i}\neq x_{j},\ i\neq j\\ \mbox{\tiny aligned counterclockwise}}\right\},
\end{equation*}
where $\del D=\overline{D}\backslash D$ is the boundary of $D$.

\begin{defn}
Let $\gamma\in (0,2)$, $N\in\mbb{Z}_{\ge 0}$ and
let $(D,h_{D},(x^{(1)}_{D},\dots, x^{(N+1)}_{D}))$ be a triple consisting of
\begin{itemize}
\item A simply connected domain $D\subsetneq \mbb{C}$.
\item A distribution $h_{D} \in C_{\nabla}^{\infty}(D)^{\prime}$.
\item An $(N+1)$-tuple of ordered boundary points $(x^{(1)}_{D},\dots, x^{(N+1)}_{D})\in\mrm{Conf}^{<}_{N+1}(\del D)$.
\end{itemize}
Triples $(D_{1},h_{D_{1}},(x^{(1)}_{D_{1}},\dots, x^{(N+1)}_{D_{1}}))$ and $(D_{2},h_{D_{2}},(x^{(1)}_{D_{2}},\dots, x^{(N+1)}_{D_{2}}))$
are said to be $\gamma$-equivalent if there exists a conformal equivalence $\psi:D_{1}\to D_{2}$ such that
$\psi(x^{(i)}_{D_{1}})=x^{(i)}_{D_{2}}$, $i=1,\dots, N+1$ and the following identity among distributions holds:
\begin{equation*}
	h_{D_{1}}=h_{D_{2}} \circ\psi +Q\log|\pr{\psi}|,
\end{equation*}
where $Q=\frac{2}{\gamma}+\frac{\gamma}{2}$.
\end{defn}

\begin{defn}[Pre-quantum surface with marked boundary points]
Let $\gamma\in (0,2)$, $N\in\mbb{Z}_{\ge 0}$.
A {\it $\gamma$-pre-quantum surface with $N+1$ marked boundary points} is a $\gamma$-equivalence class of triples
$(D,h_{D},(x^{(1)}_{D},\dots, x^{(N+1)}_{D}))$ of simply connected domains $D\subsetneq \mbb{C}$,
distributions $h_{D}\in C_{\nabla}^{\infty}(D)^{\prime}$,
and ordered boundary points $(x^{(1)}_{D},\dots, x^{(N+1)}_{D})\in\mrm{Conf}^{<}_{N+1}(\del D)$.
We denote the $\gamma$-equivalence class of $(D,h_{D},(x^{(1)}_{D},\dots, x^{(N+1)}_{D}))$ as $[D,h_{D},(x^{(1)}_{D},\dots, x^{(N+1)}_{D})]_{\gamma}$
and write the collection of $\gamma$-pre-quantum surfaces with $N+1$ marked boundary points as $\mcal{S}_{\gamma,N+1}$.
\end{defn}
We will give the construction of $\mcal{S}_{\gamma ,N+1}$ as an orbifold in Appendix \ref{app:construction}.

\begin{defn}[Quantum surface with marked boundary points]
\label{def:qsurf_boundary}
Let $\gamma\in (0,2)$ and $N\in\mbb{Z}_{\ge 0}$.
A $\gamma$-quantum surface with $N+1$ marked boundary points is a probability measure on $\mcal{S}_{\gamma,N+1}$.
Equivalently, a $\gamma$-quantum surface with $N+1$ marked boundary points is a collection of triples $(D,H_{D},(X_{D}^{(1)},\dots, X_{D}^{(N+1)}))$,
where $D\subsetneq\mbb{C}$ is a simply connected domain
and $(H_{D},(X^{(1)}_{D},\dots, X^{(N+1)}_{D}))$ is a $(C_{\nabla}^{\infty}(D)^{\prime}\times \mrm{Conf}^{<}_{N+1}(\del D))$-valued random field
subject to the condition that, for all simply connected domains $D_{1},D_{2}\subsetneq \mbb{C}$ and conformal equivalences $\psi:D_{1}\to D_{2}$,
\begin{equation*}
	(H_{D_{1}},(X^{(1)}_{D_{1}},\dots, X^{(N+1)}_{D_{1}}))\overset{\mrm{(law)}}{=}(H_{D_{2}}\circ \psi+Q\log |\pr{\psi}|,(\psi^{-1}(X^{(1)}_{D_{2}}),\dots,\psi^{-1}(X^{(N+1)}_{D_{2}})))
\end{equation*}
holds, where $Q=\frac{2}{\gamma}+\frac{\gamma}{2}$.
We write this collection as $[D,H_{D},(X^{(1)}_{D},\dots, X^{(N+1)}_{D})]_{\gamma}$.
\end{defn}

The relevant example of $\gamma$-quantum surfaces with marked boundary points in the present paper is of the {\it standard type} defined as follows:
We consider the following space:
\begin{equation}
\label{eq:def_rot_space}
	\widetilde{\mcal{S}^{\mrm{Rot}}_{\gamma,N+1}}(\mbb{H}):= C_{\nabla}^{\infty}(\mbb{H})^{\prime}\times \mrm{Conf}^{<}_{N}(\mbb{R}),
\end{equation}
where $\mrm{Conf}^{<}_{N}(\mbb{R}):=\{(x_{1},\dots, x_{N})\in \mbb{R}^{N}|x_{1}<\cdots <x_{N}\}$.
Another presentation of this space contained in Appendix \ref{app:construction} will motivate the superscript ``Rot".
Suppose that a probability space $\left(\mrm{Conf}^{<}_{N}(\mbb{R}),\mcal{F}_{N},\mbb{P}_{N}\right)$ is given.
It is equivalent to a random $N$-point configuration $\bm{X}=(X_{1},\dots, X_{N})$ on $\mbb{R}$
defined by $X_{i}:\mrm{Conf}^{<}_{N}(\mbb{R})\to \mbb{R}$; $(x_{1},\dots,x_{N})\mapsto x_{i}$, $i=1,\dots, N$.
Let $\alpha=(\alpha_{1},\cdots,\alpha_{N})$ be an $N$-tuple of real numbers.
For given $N$ points $\bm{x}=(x_{1},\dots ,x_{N})\in \mrm{Conf}^{<}_{N}(\mbb{R})$, we define a function on $\mbb{H}$
\begin{equation}
	\label{eq:function_standard}
	u_{\mbb{H}}^{\bm{x},\alpha}(z)=\sum_{i=1}^{N}\alpha_{i}\log|z-x_{i}|.	
\end{equation}
Here $z$ is the standard coordinate on $\mbb{H}$ embedded in $\mbb{C}$.
Then the assignment
\begin{equation}
\label{eq:pre-standard-field}
	(H^{\bm{X},\alpha}_{\mbb{H}},\bm{X}):\Omega^{\mrm{Fr}}_{\mbb{H}}\times \mrm{Conf}^{<}_{N}(\mbb{R})\ni(\omega,\bm{x})\mapsto (H^{\mrm{Fr}}_{\mbb{H}}(\omega)+u_{\mbb{H}}^{\bm{x},\alpha},\bm{x})\in\widetilde{\mcal{S}^{\mrm{Rot}}_{\gamma,N+1}}(\mbb{H})
\end{equation}
gives an $\widetilde{\mcal{S}^{\mrm{Rot}}_{\gamma,N+1}}(\mbb{H})$-valued random field on
$\Omega_{\mbb{H},N}^{\mrm{Fr}}:=\Omega^{\mrm{Fr}}_{\mbb{H}}\times\mrm{Conf}^{<}_{N}(\mbb{R})$
equipped with the product probability measure $\mbb{P}^{\mrm{Fr}}_{\mbb{H}}\otimes \mbb{P}_{N}$.
We denote the probability measure on $\mcal{S}_{\gamma,N+1}$ induced along the surjection
\begin{equation}
\label{eq:surj_rot_to_surf}
\pi^{\infty}_{\gamma,N+1}:\widetilde{\mcal{S}^{\mrm{Rot}}_{\gamma,N+1}}\twoheadrightarrow \mcal{S}_{\gamma,N+1},\ \ \  (h,\bm{x})\mapsto [\mbb{H},h,(\bm{x},\infty)]_{\gamma}
\end{equation}
 by $\mbb{P}_{\gamma,N+1}^{\bm{X},\alpha}$.

\begin{defn}[Standard type]
Let $\alpha=(\alpha_{1},\dots,\alpha_{N})\in\mbb{R}^{N}$
and let $\bm{X}=(X_{1},\dots, X_{N})$ be a $\mrm{Conf}^{<}_{N}(\mbb{R})$-valued random variable.
The associated probability measure $\mbb{P}^{\bm{X},\alpha}_{\gamma,N+1}$ on $\mcal{S}_{\gamma,N+1}$ constructed above
is called a $\gamma$-quantum surface with $N+1$ marked boundary points of the $(\bm{X},\alpha)$-standard type.
Equivalently, a $\gamma$-quantum surface with $N+1$ marked boundary points of the $(\bm{X},\alpha)$-standard type
is a collection of the form $[\mbb{H},H^{\bm{X},\alpha}_{\mbb{H}},(\bm{X},\infty)]_{\gamma}$, where $H^{\bm{X},\alpha}_{\mbb{H}}=H^{\mrm{Fr}}_{\mbb{H}}+u^{\bm{X},\alpha}_{\mbb{H}}$.
\end{defn}

By definition, the random variable $\bm{X}$ and the random field $H^{\mrm{Fr}}_{\mbb{H}}$ are independent.
They are combined when constructing the probability measure $\mbb{P}^{\bm{X},\alpha}_{\gamma,N+1}$
giving rise to a $C_{\nabla}^{\infty}(\mbb{H})^{\prime}$-valued random field 
$H^{\bm{X},\alpha}_{\mbb{H}}=H^{\mrm{Fr}}_{\mbb{H}}+u^{\bm{X},\alpha}_{\mbb{H}}$ depending on 
the random boundary points $\bm{X}$ taking values in $\mrm{Conf}^{<}_{N}(\mbb{R})$.

Note that the function $u_{\mbb{H}}^{\bm{x},\alpha}$ defined by Eq.\ (\ref{eq:function_standard}) is harmonic 
with logarithmic singularities at the $(N+1)$-st point $\infty$
as well as at points $x_{i}\in\mbb{R}$, $i=1,\dots, N$.
Indeed, introducing a coordinate $w=1/z$ vanishing at $\infty$,
we see that $u_{\mbb{H}}^{\bm{x},\alpha}(1/w)\sim -\sum_{i=}^{N}\alpha_{i}\log|w|$ as $w\to 0$.

In the rest of the present paper, we will abbreviate a $\gamma$-(pre-)quantum surface
with $N+1$ marked boundary points as a $\gamma$-(pre-)QS-$(N+1)$-MBPs.

\label{subsec:conformal_welding}
Now we propose the conformal welding problem for a $\gamma$-QS-$(N+1)$-MBPs of the $(\bm{X},\alpha)$-standard type
with $N\in\mbb{Z}_{\ge 1}$ as follows:
\begin{prob}
\label{prob:1}
Let $\mbb{P}^{\bm{X},\alpha}_{\gamma,N+1}$, $\gamma\in (0,2)$, $N\in\mbb{Z}_{\ge 1}$, be a $\gamma$-QS-$(N+1)$-MBPs
of the $(\bm{X},\alpha)$-standard type.
Then for each realization $(h,\bm{x})\in \widetilde{\mcal{S}_{\gamma,N+1}^{\mrm{Rot}}}(\mbb{H})$ of $(H^{\bm{X},\alpha}_{\mbb{H}},\bm{X})$, find a member
\begin{equation*}
	\left(\mbb{H}\Big\backslash\bigcup_{i=1}^{N}\eta^{(i)}(0,1],\tilde{h}, \left(\eta^{(1)}(1),\dots\eta^{(N)}(1),\infty\right)\right)
\end{equation*}
in the equivalence class $[\mbb{H},h,(\bm{x},\infty)]_{\gamma}$
such that $\eta^{(i)}:(0,1]\to\mbb{H}$, $i=1,\dots,N$ are non-intersecting slits in $\mbb{H}$
satisfying the following conditions:
\begin{enumerate}
\item 	The slits are {\it anchored} on the real axis: $\eta^{(i)}_{0}:=\lim_{t\to 0}\eta^{(i)}(t)\in\mbb{R}$, $i=1,\dots, N$.
\item	The slits are {\it seams}:
	\begin{equation*}
		\nu^{\gamma}_{\tilde{h}}(\eta^{(i)}(0,t]_{\mrm{L}})=\nu^{\gamma}_{\tilde{h}}(\eta^{(i)}(0,t]_{\mrm{R}}),\ \ t\in (0,1],\ \ i=1,\dots, N.
	\end{equation*}
	Here $\eta^{(i)}(0,t]_{\mrm{L}}$ (resp. $\eta^{(i)}(0,t]_{\mrm{R}}$) is the boundary segment lying on the left (resp. right) of
	the slit $\eta^{(i)}(0,t]$.
\end{enumerate}
\end{prob}

To explain its geometric meaning, suppose that Problem \ref{prob:1} was solved.
For each realization $H^{\bm{x},\alpha}_{\mbb{H}}(\omega)$ of $H^{\bm{X},\alpha}_{\mbb{H}}$,
let $\psi:\mbb{H}\to \mbb{H}\backslash\bigcup_{i=1}^{N}\eta^{(i)}(0,1]$ be a conformal equivalence.
Write $z_{i}^{-}$ and $z_{i}^{+}$ for the points on the real axis such that $z_{i}^{-}<x_{i}<z_{i}^{+}$ and $\psi (z_{i}^{\pm})=\eta^{(i)}_{0}$, $i=1,\dots, N$.
Then the conformal mapping $\psi$ glues the intervals $[z_{i}^{-},x_{i}]$ and $[x_{i},z_{i}^{+}]$
by means of the boundary measure $\nu^{\gamma}_{H^{\bm{x},\alpha}_{\mbb{H}}(\omega)}$ (see Fig.~\ref{fig:welding_problem}).

\begin{figure}[h]
\centering
\includegraphics[width=\hsize]{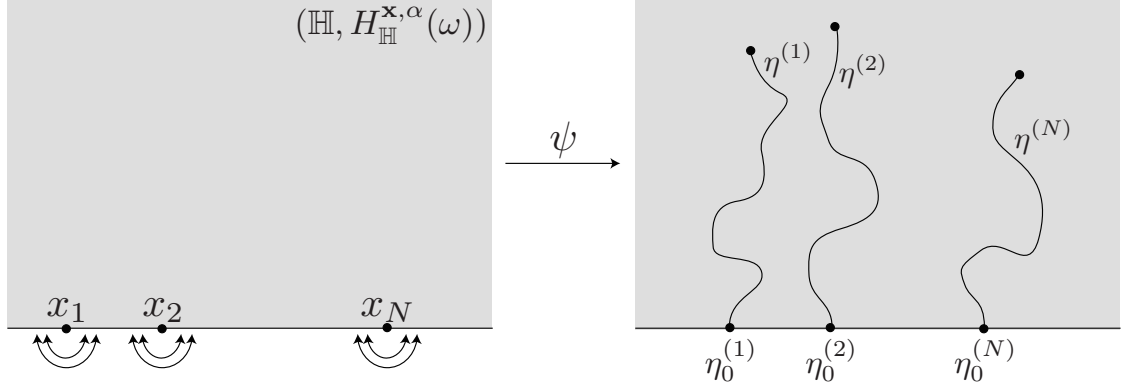}
\caption{The conformal welding problem.}
\label{fig:welding_problem}
\end{figure}

As a solution to Problem \ref{prob:1}, a statistical ensemble of slits $\{\eta^{(i)}\}_{i=1}^{N}$ is obtained.
Then, we would like to ask the probability law for the resulting curves:
\begin{prob}
\label{prob:2}
Determine the probability law for the slits $\{\eta^{(i)}\}_{i=1}^{N}$.
\end{prob}
Note that the anchor points of these curves on $\mbb{R}$ are also random variables: $\eta^{(i)}_{0}=\eta^{(i)}_{0}(\omega,\bm{x})$, $i=1,\dots, N$,
while this seems to prevent us from capturing the ensemble of slits.
Thus we also set the following problem as a sub-problem of Problem \ref{prob:2}:
\begin{prob}
\label{prob:3}
Find a $\gamma$-QS-$(N+1)$-MBPs of the $(\bm{X},\alpha)$-standard type such that the anchor points $\{\eta^{(i)}_{0}\}_{i=1}^{N}$ are deterministic.
\end{prob}

In the case of $N=1$, this problem reduces to the one addressed by Sheffield \cite{Sheffield2016}.
The space $\mcal{S}_{\gamma,N+1}$ is fibered over $\mrm{Aut}(\mbb{H})\backslash\mrm{Conf}^{<}_{N+1}(\del\mbb{H})$
by forgetting the distributions (see Appendix \ref{app:construction}).
Owing to the fact that the space $\mrm{Aut}(\mbb{H})\backslash\mrm{Conf}^{<}_{2}(\del\mbb{H})=\{[0,\infty]\}$ consists of a single element,
Problem \ref{prob:3} becomes trivial in this case.
We could say that Sheffield \cite{Sheffield2016} addressed Problems \ref{prob:1} and \ref{prob:2} for a $\gamma$-QS-$2$-MBP
of the $(\bm{X},\alpha)$-standard type with $\alpha_{1}=\frac{2}{\gamma}$
and that he found a one-parameter family of solutions to Problem \ref{prob:1};
the reverse flow of an SLE gives the required conformal equivalences.
Consequently, the resulting curve obeys the probability law of the one for the SLE of parameter $\kappa=\gamma^{2}$
as the solution to Problem \ref{prob:2}.

\subsubsection*{\bf Result}
To present a result, let us recall the Dyson model.
In the usual convention, the Dyson model of parameter $\beta >0 $ is the system of SDEs on
$\{\bm{X}^{\mrm{D}_{\beta}}_{t}=(X^{\mrm{D}_{\beta}(1)}_{t},\dots, X^{\mrm{D}_{\beta}(N)}_{t})\in\mbb{R}^{N}:t\ge 0\}$ \cite{Dyson1962,Katori2015}:
\begin{equation}
	\label{eq:Dyson_model}
	dX^{\mrm{D}_{\beta}(i)}_{t}=dB^{(i)}_{t}+\frac{\beta}{2}\sum_{\substack{j=1\\j\neq i}}^{N}\frac{1}{X^{\mrm{D}_{\beta}(i)}_{t}-X^{\mrm{D}_{\beta}(j)}_{t}}dt,\ \ t\ge 0,\ \ i=1,\dots, N.
\end{equation}
It is known that when $\beta \ge 1$, the Dyson model with an arbitrary finite number of particles $N \in\mbb{Z}_{\ge 2}$
has a strong and pathwise unique non-colliding solution for general initial conditions \cite{RogersShi1993,CepaLepingle1997,GraczykMalecki2013, GraczykMalecki2014}.
The non-colliding condition $\beta \ge 1$ for the Dyson model corresponds to $\kappa \in (0, 8 ] $ through the relation $\beta=8/\kappa$.
See Section \ref{subsec:dyson_model} for more detail on the relations among parameters.

We obtain a solution to Problems \ref{prob:1} to \ref{prob:3} for $\gamma$-QSs-$(N+1)$-MBPs as follows:
\begin{thm}
\label{thm:answer_welding}
Let $\gamma=\sqrt{\kappa}\in (0,2)$ and $N\in\mbb{Z}_{\ge 1}$.
Suppose that $\{\bm{X}_{t}=(X_{t}^{(1)},\dots, X_{t}^{(N)}):t\ge 0\}$ is a time change of the Dyson model
$\{\bm{X}^{\mrm{D}_{8/\kappa}}_{\kappa t}:t\ge 0\}$
starting at a deterministic initial state $\bm{X}_{0}=\bm{x}\in\mrm{Conf}^{<}_{N}(\mbb{R})$.
Then, at each time $T\in (0,\infty)$, the conformal welding problem for $\left[\mbb{H}, H^{\bm{X}_{T},\alpha}_{\mbb{H}},(\bm{X}_{T},\infty)\right]_{\gamma}$
with $(\alpha_{1},\dots,\alpha_{N})=(\frac{2}{\gamma},\dots,\frac{2}{\gamma})$ is solved as follows:
\begin{enumerate}
\item 	The solution of the Loewner equation (\ref{eq:multiple_Loewner}) driven by the time change of the Dyson model 
		$\{\bm{X}^{\mrm{D}_{8/\kappa}}_{\kappa t}:0\le t \le T\}$ gives a solution to Problem \ref{prob:1}.
		In other words, $g_{T}^{-1}:\mbb{H} \to \mbb{H}^{\eta}_{T}$ is the desired conformal equivalence.
\item 	The probability law for resulting slits $\{\eta^{(i)}\}_{i=1}^{N}$ is the one for the multiple SLE$_{\kappa}$.
		This gives a solution to Problem \ref{prob:2}.
\item 	Problem \ref{prob:3} is answered positively with $\bm{\eta}_{0}=\bm{x}$ a.s.
\end{enumerate}
\end{thm}
This theorem will be proved in Sect.\ \ref{sect:coupling_reverse}.

\subsection{Flow line problem}
Another topic in random geometry that stems from GFF is the imaginary geometry \cite{MillerSheffield2016a,MillerSheffield2016b,MillerSheffield2016c,MillerSheffield2017},
which sees the flow line of the vector field $e^{\sqrt{-1}H/\chi}$, where $H$ is a $C^{\infty}_{0}(D)^{\prime}$-valued random field
and $\chi>0$ is a parameter.
Let us temporarily suppose that $h$ were a smooth function on $D\subsetneq \mbb{C}$. 
Then $e^{\sqrt{-1}h/\chi}$ is a smooth vector field and its flow line $\eta:[0,\infty)\to D$ starting at $x_{0}\in D$
is defined as the solution of the ordinary differential equation
\begin{equation*}
	\frac{d\eta(t)}{dt}=e^{\sqrt{-1}h(\eta(t))/\chi},\ \ t\ge 0,\ \ \ \eta(0)=x_{0}\in D.
\end{equation*}
For another simply connected domain $\tilde{D}$ and a conformal equivalence $\psi:\tilde{D}\to D$,
we can consider the pull-back $\tilde{\eta}=\psi^{-1}\circ \eta$ of the flow line $\eta$ by $\psi$.
Then $\tilde{\eta}$ satisfies the following differential equation:
\begin{equation*}
	\frac{d\tilde{\eta}(t)}{dt}=\frac{1}{|\pr{\psi}(\tilde{\eta}(t))|}e^{\sqrt{-1}(h\circ \psi-\chi\arg \pr{\psi})(\tilde{\eta}(t))/\chi},\ \ t\ge 0.
\end{equation*}
When we adopt a time change $\tilde{\eta}_{f}(t)=\tilde{\eta}(f(t))$ with $f(t)=\int_{0}^{t}|\pr{\psi}(\tilde{\eta}(s))|ds$,
we see that
\begin{equation*}
	\frac{d\tilde{\eta}_{f}(t)}{dt}=e^{\sqrt{-1}(h\circ\psi-\chi\arg\pr{\psi})(\tilde{\eta}_{f}(t))/\chi},\ \ t\ge 0.
\end{equation*}
Since a time reparametrization does not change the whole curve,
we can say that the domains $D$ with a smooth function $h$
and $\tilde{D}$ with $h\circ\psi-\chi \arg\pr{\psi}$ are
equivalent as long as flow lines of vector fields $e^{\sqrt{-1}h/\chi}$ and $e^{\sqrt{-1}(h\circ\psi-\chi\arg\pr{\psi})/\chi}$ are concerned.
Interestingly, this equivalence relation also makes sense even when we work with a $C^{\infty}_{0}(D)^{\prime}$-valued random field $H$
instead of a smooth function \cite{MillerSheffield2016a,MillerSheffield2016b,MillerSheffield2016c,MillerSheffield2017}.
\begin{defn}
A {\it $\chi$-imaginary surface} ($\chi$-IS) is an equivalence class of pairs $(D,H)$ of simply connected domains $D\subsetneq\mbb{C}$
and $C^{\infty}_{0}(D)^{\prime}$-valued random fields $H$ under the equivalence relation
\begin{equation*}
	(D,H)\sim_{\chi} (\tilde{D},\tilde{H}):=(\psi^{-1}(D),H\circ \psi -\chi\arg \pr{\psi}),
\end{equation*}
where $\tilde{D}\subsetneq\mbb{C}$ is another simply connected domain 
and $\psi:\tilde{D}\to D$ is a conformal equivalence.
\end{defn}

It has been shown in ~\cite{MillerSheffield2016a,Sheffield2016} that  for a $\chi$-IS
$(\mbb{H},H^{\mrm{Dir}}_{\mbb{H}}-\frac{2}{\sqrt{\kappa}} \arg(\cdot))$
with $\chi=\frac{2}{\sqrt{\kappa}}-\frac{\sqrt{\kappa}}{2}$,
the flow line starting at the origin can be identified with the curve for the SLE$_{\kappa}$.
Note that the $\arg$ function is a harmonic function on $\mbb{H}$ with the boundary value
changes at the origin and infinity by $\pi$.
In this sense, above imaginary surface is said to have {\it boundary condition changing points} (BCCPs) at the origin and infinity.

We settle the {\it flow line problem} for a $\chi$-IS with more BCCPs than two as follows:
Let $\bm{x}=(x_{1},\dots, x_{N})\in\mrm{Conf}^{<}_{N}(\mbb{R})$ and $\beta_{1},\dots,\beta_{N}$ be real numbers.
The $C^{\infty}_{0}(\mbb{H})^{\prime}$-valued random field
\begin{equation*}
	H_{\mbb{H}}^{\bm{x},\beta,\mfrak{I}}=H^{\mrm{Dir}}_{\mbb{H}}-\sum_{i=1}^{N}\beta_{i}\arg (\cdot-x_{i})
\end{equation*}
on $\mbb{H}$ defines a $\chi$-IS with BCCPs $(x_{1},\dots,x_{N},\infty)$,
whose boundary value has discontinuity at $x_{i}$ by $\pi\beta_{i}$, $i=1,\dots, N$
and at $\infty$ by $-\pi\sum_{i=1}^{N}\beta_{i}$.
Here $\mfrak{I}$ stands for {\it imaginary}.
\begin{prob}
\label{prob:4}
Determine the probability law for multiple flow lines of a $\chi$-IS with $N+1$ BCCPs 
($\chi$-IS-$(N+1)$-BCCPs) $H_{\mbb{H}}^{\bm{x},\beta,\mfrak{I}}$
starting at boundary points $\bm{x}\in\mrm{Conf}^{<}_{N}(\mbb{R})$.
\end{prob}

This problem is solved as follows:
\begin{thm}
\label{thm:answer_flowline}
Let $\kappa\in (0,4]$ and $N\in\mbb{Z}_{\ge 1}$.
The flow line problem for the $\chi$-IS-$(N+1)$-BCCPs $(\mbb{H},H^{\bm{x},\beta,\mfrak{I}}_{\mbb{H}})$ 
with $\chi=\frac{2}{\sqrt{\kappa}}-\frac{\sqrt{\kappa}}{2}$
is solved for any boundary points $\bm{x}=(x_{1},\dots, x_{N})\in \mrm{Conf}^{<}_{N}(\mbb{R})$
if $\beta=(\beta_{1},\dots,\beta_{N})$ is given by $\beta_{i}=\frac{2}{\sqrt{\kappa}}$, $i=1,\dots, N$.
The probability law of the flow lines is given by the multiple SLE$_{\kappa}$ driven by
the time change of the Dyson model $\set{\bm{X}^{\mrm{D}_{8/\kappa}}_{\kappa t}:t\ge 0}$.
\end{thm}

This theorem will be proved in Sect.\ \ref{sect:coupling_forward}.
In the case when $\kappa=4$, a related problem was addressed in ~\cite{PeltolaWu2019},
where the connection probability of the multiple level lines for the GFF was investigated.

\section{Proof of Theorem \ref{thm:answer_welding}}
\label{sect:coupling_reverse}

\subsection{Reduction of Theorem \ref{thm:answer_welding}}
To solve Problems \ref{prob:1} to \ref{prob:3},  we define the {\it cutting operation} on 
$\widetilde{\mcal{S}_{\gamma,N+1}^{\mrm{Rot}}}(\mbb{H})$-valued random fields associated with a multiple SLE.
Let $(H,\bm{X})$ be an $\widetilde{\mcal{S}_{\gamma,N+1}^{\mrm{Rot}}}(\mbb{H})$-valued random field (see the right picture in Fig.~\ref{fig:solution_Loewner}).
Suppose that we have an $N$-tuple of interacting stochastic processes $\{\bm{X}_{t}=(X^{(1)}_{t},\dots, X^{(N)}_{t})\in\mbb{R}^{N}:t\ge 0\}$ 
with the initial conditions $X^{(i)}_{0}=X_{i}$, $i=1,\dots, N$, which is conditionally independent of $(H,\bm{X})$.
We assume that it determines random slits $\{\eta^{(i)}\}_{i=1}^{N}$ through the correspondence in Theorem \ref{thm:multple_Loewner}.
Then these slits are anchored at the initial marked boundary points $\bm{X}$, i.e., $\lim_{t\to 0}\eta^{(i)}(t)=X_{i}$, $i=1,\dots, N$, a.s.
The configuration space for $\{\bm{X}_{t}:t\ge 0\}$ is identified with $\mrm{Conf}^{<}_{N}(\mbb{R})\times [0,\infty)$.
We denote the probability law on the space $\mrm{Conf}^{<}_{N}(\mbb{R})\times [0,\infty)$
which governs $\{\bm{X}_{t}:t\ge 0\}$ by $\mbb{P}^{\mrm{SLE}}$, since it also governs the multiple SLE given in the form (\ref{eq:multiple_Loewner}).
Then at each time $T\in [0,\infty)$, $\mbb{P}^{\mrm{SLE}}$ induces a probability measure on $\mrm{Conf}^{<}_{N}(\mbb{R})$ by
\begin{equation*}
	\mbb{P}^{\mrm{SLE}}_{T}(dx_{1},\dots,dx_{N})=\mbb{P}^{\mrm{SLE}}(X^{(1)}_{T}\in dx_{1},\dots,X^{(N)}_{T}\in dx_{N}).
\end{equation*}

\begin{figure}
\centering
\includegraphics[width=\hsize]{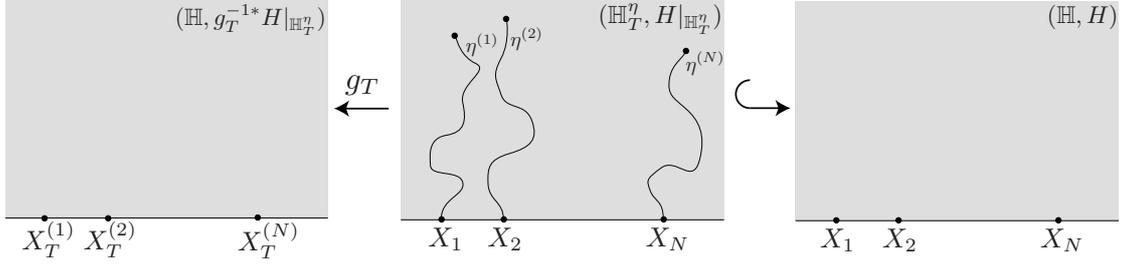}
\caption{Cutting operation by multiple SLE.}
\label{fig:solution_Loewner}
\end{figure}

We fix a time $T\in [0,\infty)$ and restrict the random distribution $H$ on the domain $\mbb{H}^{\eta}_{T}\subsetneq \mbb{H}$
to obtain a new $\gamma$-QS-$(N+1)$-MBPs (see the middle picture in Fig.~\ref{fig:solution_Loewner})
\begin{equation}
	\label{eq:quantum_surface_restricted}
	\left[\mbb{H}^{\eta}_{T},(\eta^{(1)}(T),\dots, \eta^{(N)}(T),\infty),H|_{\mbb{H}^{\eta}_{T}}\right]_{\gamma}.
\end{equation}
It is manifest from this construction that we have
\begin{equation*}
	\nu^{\gamma}_{H|_{\mbb{H}^{\eta}_{T}}}(\eta^{(i)}(0,t]_{\mrm{L}})=\nu^{\gamma}_{H|_{\mbb{H}^{\eta}_{T}}}(\eta^{(i)}(0,t]_{\mrm{R}}),\ \ t\in [0,T],\ \ i=1,\dots, N,\ \ \mrm{a.s.}
\end{equation*}

We define an $\widetilde{\mcal{S}_{\gamma,N+1}^{\mrm{Rot}}}(\mbb{H})$-valued random field (see the left picture in Fig.~\ref{fig:solution_Loewner})
\begin{equation}
	\label{eq:qsurf_obvious}
	\mfrak{A}^{(\bm{X}_{t}:0\le t\le T)}(H,\bm{X}):=\left(g_{T}^{-1 \ast}H|_{\mbb{H}^{\eta}_{T}}, (X^{(1)}_{T},\dots, X^{(N)}_{T})\right),
\end{equation}
where we wrote $g_{T}^{-1\ast}H|_{\mbb{H}^{\eta}_{T}}=H|_{\mbb{H}^{\eta}_{T}}\circ g_{T}^{-1}+Q\log |g_{T}^{-1\prime}|$.
Then, uniformizing the domain $\mbb{H}^{\eta}_{T}$ to $\mbb{H}$ by the conformal mapping $g_{T}$,
we can find that the $\gamma$-QS-$(N+1)$-MBPs (\ref{eq:quantum_surface_restricted}) coincides with
$\pi^{\infty}_{\gamma,N+1}(\mfrak{A}^{(\bm{X}_{t}:0\le t\le T)}(H,\bm{X}))$.
We call this assignment $\mfrak{A}^{(\bm{X}_{t}:0\le t\le T)}$ of the $\widetilde{\mcal{S}_{\gamma,N+1}^{\mrm{Rot}}}(\mbb{H})$-valued random field (\ref{eq:qsurf_obvious})
to a given $\widetilde{\mcal{S}_{\gamma,N+1}^{\mrm{Rot}}}(\mbb{H})$-valued random field $(H,\bm{X})$ the {\it cutting operation} associated with the multiple SLE driven by $\{\bm{X}_{t}:0\le t\le T\}$.

Notice that for a quantum surface $\pi^{\infty}_{\gamma,N+1}(\mfrak{A}^{(\bm{X}_{t}:0\le t\le T)}(H,\bm{X}))$ 
obtained by the cutting operation associated with the multiple SLE, the versions of Problems \ref{prob:1} to \ref{prob:3} are easily solved.
Indeed, the mapping $g_{T}$ gives the desired conformal equivalence to solve Problem \ref{prob:1}.
Problem \ref{prob:3} can also be answered:
Since the Loewner theory ensures that the slits $\{\eta^{(i)}\}_{i=1}^{N}$ are anchored at $\bm{X}$,
i.e., $\bm{\eta}_{0}:=(\eta^{(1)}_{0},\dots, \eta^{(N)}_{0})=\bm{X}$, a.s.,
the anchor points $\bm{\eta}_{0}$ are deterministic if and only if the initial configuration $\bm{X}$ is deterministic.
In this case, the probability law of the resulting slits $\{\eta^{(i)}\}_{i=1}^{N}$ is given by the law of the multiple SLE,
answering Problem \ref{prob:2}.

Therefore, if a $\gamma$-QS-$(N+1)$-MBPs obtained by the composition of the cutting operation and the surjection $\pi^{\infty}_{\gamma,N+1}$
is of the $(\widetilde{\bm{X}},\alpha)$-standard type for some $\mrm{Conf}^{<}_{N}(\mbb{R})$-valued random variable $\widetilde{\bm{X}}$,
the conformal welding problem for it is solved in the above arguments.
The following proposition gives such examples.
\begin{prop}
\label{prop:stationarity_welding}
Let $(H_{\mbb{H}}^{\bm{X},\alpha},\bm{X})$ be an $\widetilde{\mcal{S}_{\gamma,N+1}^{\mrm{Rot}}}(\mbb{H})$-valued random field
introduced in (\ref{eq:pre-standard-field}) with $(\alpha_{1},\dots, \alpha_{N})=(\frac{2}{\gamma},\dots, \frac{2}{\gamma})$
and assume that $\{\bm{X}_{t}=(X^{(1)}_{t},\dots, X^{(N)}_{t}):t\ge 0\}$ solves the set of SDEs (\ref{eq:driving_process_Schramm})
with $\kappa=\gamma^{2}$ and
\begin{equation}
\label{eq:function_canonical}
	F^{(i)}(\bm{x})=\sum_{\substack{j=1\\j\neq i}}^{N}\frac{4}{x_{i}-x_{j}},\ \ i=1,\dots, N,
\end{equation}
starting at $\bm{X}$.
Then, for an arbitrary $T\in [0,\infty)$, we have
\begin{equation}
\label{eq:qsurf_T_standard}
	\mfrak{A}^{(\bm{X}_{t}:0\le t\le T)}(H_{\mbb{H}}^{\bm{X},\alpha},\bm{X})\overset{\mrm{(law)}}{=}\left(H^{\bm{X}_{T},\alpha}_{\mbb{H}},\bm{X}_{T}\right)
\end{equation}
as $\widetilde{\mcal{S}_{\gamma,N+1}^{\mrm{Rot}}}(\mbb{H})$-valued random fields.
In particular, $\pi^{\infty}_{\gamma,N+1}(\mfrak{A}^{(\bm{X}_{t}:0\le t\le T)}(H,\bm{X}))$
is a $\gamma$-QS-$(N+1)$-MBPs of the $(\bm{X}_{T},\alpha)$-standard type.
\end{prop}

Note that Theorem \ref{thm:answer_welding} is concluded from of Proposition \ref{prop:stationarity_welding} owing to the arguments followed by Proposition \ref{prop:stationarity_welding} above. Indeed the stochastic process $\set{\bm{X}_{t}:t\ge 0}$ solving (\ref{eq:driving_process_Schramm}) with (\ref{eq:function_canonical}) is just the desired time change of the Dyson model. To make the argument more precise, one only has to notice that a multiple SLE driven by such a Dyson model is absolutely continuous with respect to multiple of independent SLEs \cite{Graham2007}, and therefore, the argument in \cite{Sheffield2016} can be applied.

In the subsequent subsections, we prove Proposition \ref{prop:stationarity_welding}.
It suffices to prove that $H^{\bm{X}_{T},\alpha}_{\mbb{H}}$ 
and $g_{T}^{-1\ast}H^{\bm{X},\alpha}_{\mbb{H}}|_{\mbb{H}^{\eta}_{T}}$ obey the same probability law.
The central idea is to interpolate these two $C_{\nabla}^{\infty}(\mbb{H})^{\prime}$-valued random fields by a single stochastic process and show its stationarity.
The cutting operation $\mfrak{A}^{(\bm{X}_{t}:0\le t\le T)}$ indeed defines a candidate of such an interpolating stochastic process,
but it turns out that the reverse flow behaves much better.
Before proceeding, let us introduce the reverse flow of a multiple Loewner chain needed in our proof.

\subsection{Reverse flow of a multiple Loewner chain}
Let $\{\bm{X}_{t}=(X^{(1)}_{t},\dots, X^{(N)}_{t}):t\ge 0\}$ be a set of continuous functions of $t$ that drives the Loewner equation (\ref{eq:multiple_Loewner}).
We assume that given $\{\bm{X}_{t}:t\ge 0\}$, the Lowener equation (\ref{eq:multiple_Loewner}) has a unique solution
and determines an $N$-tuple of slits $\{\eta^{(i)}\}_{i=1}^{N}$.
We fix a time $T\in (0,\infty)$ and set $Y_{T;t}^{(i)}:=X^{(i)}_{T-t}$, $t\in [0,T]$, $i=1,\dots, N$.
The reverse flow of the Loewner chain is defined as the solution of
\begin{equation}
\label{eq:reverse_multiple_SLE}
	\frac{d}{dt}f^{T}_{t}(z)=-\sum_{i=1}^{N}\frac{2}{f^{T}_{t}(z)-Y_{T;t}^{(i)}},\ \ t\in [0,T],\ \ \ f^{T}_{0}(z)=z\in\mbb{H}.
\end{equation}

\begin{lem}
\label{lem:reverse_flow_is_inverse_map}
The identity $f^{T}_{T}=g_{T}^{-1}$ holds,
where $g_{T}^{-1}:\mbb{H}\to\mbb{H}^{\eta}_{T}$ is the inverse map of the uniformizing map $g_{T}:\mbb{H}^{\eta}_{T}\to\mbb{H}$
that solves the Loewner equation (\ref{eq:multiple_Loewner}).
\end{lem}
\begin{proof}
Set $u^{T}_{t}(z):=f^{T}_{T-t}(z)$, $z\in\mbb{H}$, $t\in [0,T]$.
Then it satisfies
\begin{equation*}
	\frac{d}{dt}u^{T}_{t}(z)=\sum_{i=1}^{N}\frac{2}{u^{T}_{t}(z)-X^{(i)}_{t}},\ \ t\in [0,T],\ \ \ u^{T}_{0}(z)=f^{T}_{T}(z),\ \ z\in \mbb{H}.
\end{equation*}
Since we have assumed that the multiple Loewner equation (\ref{eq:multiple_Loewner}) has a unique solution,
this implies that $u^{T}_{t}(z)=g_{t}(f^{T}_{T}(z))$.
Indeed, both sides satisfy the same differential equation with the same initial condition.
In particular, at time $t=T$, $u^{T}_{T}(z)=f^{T}_{0}(z)=z=g_{T}(f^{T}_{T}(z))$, $z\in\mbb{H}$,
implying that $f^{T}_{T}=g_{T}^{-1}$.
\end{proof}

The reverse flow of the multiple SLE can also be formulated in connection to CFT as described in Appendix \ref{app:driving_processes}.

\subsection{Interpolation of random fields}
We assume that the set of driving processes $\{\bm{X}_{t}=(X^{(1)}_{t},\dots, X^{(N)}_{t}):t\ge 0\}$ is given by the system of SDEs (\ref{eq:driving_process_Schramm}).
For a fixed $T\in (0,\infty)$, we set the time reversed process $Y^{(i)}_{T;t}=X^{(i)}_{T-t}$, $t\in [0,T]$, $i=1,\dots, N$,
and let $\{f^{T}_{t}(\cdot):t\in [0,T]\}$ be the reverse flow defined in Eq.\ (\ref{eq:reverse_multiple_SLE}) driven by $\{\bm{Y}_{T;t}=(Y^{(1)}_{T:t},\dots, Y^{(N)}_{T;t}): t\in [0,T]\}$.
From Lemma \ref{lem:reverse_flow_is_inverse_map}, we can conclude that
$g_{T}^{-1\ast}H|_{\mbb{H}^{\eta}_{T}}=f_{T}^{T\ast}H|_{\mbb{H}^{\eta}_{T}}$ a.s.

Let us define a stochastic process
\begin{equation*}
	\mfrak{n}_{t}(z)=u^{\bm{Y}_{T;t},(2/\gamma,\dots, 2/\gamma)}_{\mbb{H}}(z)=\frac{2}{\gamma}\sum_{i=1}^{N}\log|z-Y_{T;t}^{(i)}|,\ \ t\in [0,T],\ \ z\in \mbb{H}.
\end{equation*}
We also consider $H^{\mrm{Fr}}_{\mbb{H}}$ which is independent of $\{B_{t}^{(i)}:t\ge 0\}_{i=1}^{N}$.
Then we see that, at each time $t\in [0,T]$, $[\mbb{H},H^{\mrm{Fr}}_{\mbb{H}}+\mfrak{n}_{t},(\bm{Y}_{T;t},\infty)]_{\gamma}$ is a $\gamma$-QS-$(N+1)$-MBPs
of the $(\bm{Y}_{T;t},\alpha)$-standard type with $(\alpha_{1},\dots, \alpha_{N})=(\frac{2}{\gamma},\dots, \frac{2}{\gamma})$.
We set
\begin{equation*}
	\mfrak{h}_{t}(z):=\mfrak{n}_{t}(f^{T}_{t}(z))+Q\log |f^{T\prime}_{t}(z)|,\ \ z\in \mbb{H},\ \ t\in [0,T]
\end{equation*}
with $Q=\frac{2}{\gamma}+\frac{\gamma}{2}$
and set
\begin{equation*}
	\mfrak{p}_{t}=\mfrak{h}_{t}+H_{\mbb{H}}^{\mrm{Fr}}\circ f^{T}_{t},\ \  t\in [0,T].
\end{equation*}
Then we can see that the stochastic process $\mfrak{p}_{t}$, $t\in [0,T]$ interpolates two $C_{\nabla}^{\infty}(H)^{\prime}$-valued random fields so that
 $\mfrak{p}_{0}=H^{\bm{X}_{T},\alpha}_{\mbb{H}}$ and $\mfrak{p}_{T}=f_{T}^{T\ast}H^{\bm{X},\alpha}_{\mbb{H}}|_{\mbb{H}^{\eta}_{T}}$.

\subsection{Stationarity of the stochastic process}
We claim that $\mfrak{p}_{0}$ and $\mfrak{p}_{T}$ obey the same probability law.
The proof relies on the following key lemmas:
\begin{lem}
\label{lem:key_martingale}
The stochastic process $\mfrak{h}_{t}(z)$, $z\in\mbb{H}$, $t\in [0,T]$ is a local martingale with increment
\begin{equation}
	\label{eq:interpolating_increment_welding}
	d\mfrak{h}_{t}(z)=\sum_{i=1}^{N}\mrm{Re}\frac{-2}{f^{T}_{t}(z)-Y_{T;t}^{(i)}}dB_{t}^{(i)},\ \ z\in\mbb{H},\ \ t\in [0,T],
\end{equation}
 if $\kappa=\gamma^{2}$ and the functions $\{F^{(i)}(\bm{x})\}_{i=1}^{N}$ are chosen as (\ref{eq:function_canonical}).
\end{lem}

\begin{proof}
Note that $\mfrak{h}_{t}(z)$ is the real part of
\begin{equation*}
	\mfrak{h}^{\ast}_{t}(z)=\frac{2}{\gamma}\sum_{i=1}^{N}\log (f^{T}_{t}(z)-Y_{T;t}^{(i)})+Q \log f^{T\prime}_{t}(z),\ \ z\in\mbb{H},\ \ t\in [0,T].
\end{equation*}
We will show that $\mfrak{h}^{\ast}_{t}(z)$, $z\in\mbb{H}$, $t\in [0,T]$ is a local martingale if $\kappa=\gamma^{2}$ and 
the functions $\{F^{(i)}(\bm{x})\}_{i=1}^{N}$ are chosen as (\ref{eq:function_canonical}).
Owing to the time reversibility of the Brownian motions, the set of time reversed driving processes $\{\bm{Y}_{T;t}:0\le t\le T\}$
solves the following system of SDEs:
\begin{equation}
	\label{eq:driving_process_reverse}
	dY^{(i)}_{T;t}=\sqrt{\kappa}dB^{(i)}_{t}-F^{(i)}(\bm{Y}_{T;t})dt,\ \ t\in [0,T],\ \ i=1,\dots, N.
\end{equation}
By Eqs.\ (\ref{eq:reverse_multiple_SLE}) and (\ref{eq:driving_process_reverse}), It{\^o}'s formula gives
\begin{align*}
	d\log (f^{T}_{t}(z)-Y_{T;t}^{(i)})
	&=\frac{1}{f^{T}_{t}(z)-Y_{T;t}^{(i)}}\left(\sum_{j=1}^{N}\frac{-2dt}{f^{T}_{t}(z)-Y_{T;t}^{(j)}}-\sqrt{\kappa}dB_{t}^{(i)}+F^{(i)}(\bm{Y}_{T;t})dt\right)\\
	&\hspace{15pt}-\frac{1}{2}\frac{\kappa dt}{(f^{T}_{t}(z)-Y_{T;t}^{(i)})^{2}},\\
	d\log f^{T\prime}_{t}(z)
	&=\sum_{i=1}^{N}\frac{2dt}{(f^{T}_{t}(z)-Y_{T;t}^{(i)})^{2}}.
\end{align*}
They are assembled to give the increment of $\mfrak{h}^{\ast}_{t}(z)$, $z\in\mbb{H}$,
\begin{align*}
	d\mfrak{h}^{\ast}_{t}(z)
	&=\frac{-4}{\sqrt{\kappa}}\left(\sum_{i=1}^{N}\frac{1}{f^{T}_{t}(z)-Y_{T;t}^{(i)}}\right)^{2}dt
		-\sum_{i=1}^{N}\frac{2dB_{t}^{(i)}}{f^{T}_{t}(z)-Y_{T;t}^{(i)}}
	+\frac{2}{\sqrt{\kappa}}\sum_{i=1}^{N}\frac{F^{(i)}(\bm{Y}_{T;t})dt}{f^{T}_{t}(z)-Y_{T;t}^{(i)}}\\
		&\hspace{15pt}-\sum_{i=1}^{N}\frac{\sqrt{\kappa}dt}{(f^{T}_{t}(z)-Y_{T;t}^{(i)})^{2}}
	+2Q\sum_{i=1}^{N}\frac{dt}{(f^{T}_{t}(z)-Y_{T;t}^{(i)})^{2}},\ \ t\in [0,T],
\end{align*}
where we have used the relation $\kappa=\gamma^{2}$.
It can be verified that
\begin{equation*}
		\sum_{\substack{i,j=1\\i\neq j}}^{N}\frac{1}{f^{T}_{t}(z)-Y_{T;t}^{(i)}}\frac{1}{Y_{T;t}^{(i)}-Y_{T;t}^{(j)}}
		=\frac{1}{2}\sum_{\substack{i,j=1\\i\neq j}}^{N}\frac{1}{(f^{T}_{t}(z)-Y_{T;t}^{(i)})(f^{T}_{t}(z)-Y_{T;t}^{(j)})}.
\end{equation*}
Using this, we see that the increment of $\mfrak{h}^{\ast}_{t}(z)$, $z\in\mbb{H}$ becomes
\begin{align*}
	d\mfrak{h}^{\ast}_{t}(z)=&\frac{2}{\sqrt{\kappa}}\sum_{i=1}^{N}\frac{1}{f^{T}_{t}(z)-Y^{(i)}_{T;t}}\left(F^{(i)}(\bm{Y}_{T;t})-\sum_{\substack{j=1\\ j\neq i}}^{N}\frac{4}{Y^{(i)}_{T;t}-Y^{(j)}_{T;t}}\right)dt \\
	&-\sum_{i=1}^{N}\frac{2dB_{t}^{(i)}}{f^{T}_{t}(z)-Y_{T;t}^{(i)}},\ \ t\in [0,T]
\end{align*}
and conclude that the stochastic process $\mfrak{h}^{\ast}_{t}(z)$ is a local martingale
if the functions $\{F^{(i)}(\bm{x})\}_{i=1}^{N}$ are chosen as (\ref{eq:function_canonical}).
Moreover, under such a choice of the functions $\{F^{(i)}(\bm{x})\}_{i=1}^{N}$, Eq.\ (\ref{eq:interpolating_increment_welding}) is obtained.
\end{proof}

Thus, at each $z\in\mbb{H}$, the stochastic process $\{\mfrak{h}_{t}(z):t\in [0,T]\}$ can be regarded as a Brownian motion after an appropriate time change.
In the following, we assume that the functions $\{F^{(i)}(\bm{x})\}_{i=1}^{N}$ are as (\ref{eq:function_canonical}).
By Eq.\ (\ref{eq:interpolating_increment_welding}), the cross variation between $\mfrak{h}_{t}(z)$ and $\mfrak{h}_{t}(w)$, $z,w\in\mbb{H}$ is given by
\begin{equation*}
	d\braket{\mfrak{h}(z),\mfrak{h}(w)}_{t}=\sum_{i=1}^{N}\left(\mrm{Re}\frac{2}{f^{T}_{t}(z)-Y_{T;t}^{(i)}}\right)\left( \mrm{Re}\frac{2}{f^{T}_{t}(w)-Y_{T;t}^{(i)}}\right)dt,\ \ z, w\in\mbb{H}.
\end{equation*}

\begin{lem}
Define $G^{\mrm{Fr}}_{\mbb{H}^{\eta}_{t}}(z,w):=G^{\mrm{Fr}}_{\mbb{H}}(f^{T}_{t}(z),f^{T}_{t}(w))$, $t\in [0,T]$, $z,w\in\mbb{H}$.
Then
\begin{equation*}
	d\braket{\mfrak{h}(z),\mfrak{h}(w)}_{t}=-dG^{\mrm{Fr}}_{\mbb{H}^{\eta}_{t}}(z,w),\ \ t\in [0,T],\ \ z, w\in \mbb{H}.
\end{equation*}
\end{lem}
\begin{proof}
This can be verified by direct computation.
By definition, we have
\begin{equation*}
	G_{\mbb{H}^{\eta}_{t}}^{\mrm{Fr}}(z,w)=-\log |f^{T}_{t}(z)-f^{T}_{t}(w)|-\log |f^{T}_{t}(z)-\overline{f^{T}_{t}(w)}|.
\end{equation*}
Thus its increment is computed as
\begin{align*}
	dG^{\mrm{Fr}}_{\mbb{H}^{\eta}_{t}}(z,w)
	&=-\mrm{Re}\frac{df^{T}_{t}(z)-df^{T}_{t}(w)}{f^{T}_{t}(z)-f^{T}_{t}(w)}
		-\mrm{Re}\frac{df^{T}_{t}(z)-d\overline{f^{T}_{t}(w)}}{f^{T}_{t}(z)-\overline{f^{T}_{t}(w)}}\\
	&=\sum_{i=1}^{N}\mrm{Re}\frac{-2dt}{(f^{T}_{t}(z)-Y_{T;t}^{(i)})(f^{T}_{t}(w)-Y_{T;t}^{(i)})}\\
	&\hspace{15pt}+\sum_{i=1}^{N}\mrm{Re}\frac{-2dt}{(f^{T}_{t}(z)-Y_{T;t}^{(i)})(\overline{f^{T}_{t}(w)}-Y_{T;t}^{(i)})}\\
	&=-\sum_{i=1}^{N}\left(\mrm{Re}\frac{2}{f^{T}_{t}(z)-Y_{T;t}^{(i)}}\right)\left(\mrm{Re}\frac{2}{f^{T}_{t}(w)-Y_{T;t}^{(i)}}\right)dt
\end{align*}
which is the same as $-d\braket{\mfrak{h}(z),\mfrak{h}(w)}_{t}$, $t\in [0,T]$, $z,w\in\mbb{H}$.
\end{proof}

For a test function $\rho\in C_{\nabla}^{\infty}(\mbb{H})$, we have
\begin{equation*}
	d\braket{(\mfrak{h},\rho),(\mfrak{h},\rho)}_{t}=-dE^{\mrm{Fr}}_{t}(\rho),	
\end{equation*}
where
\begin{equation*}
	E^{\mrm{Fr}}_{t}(\rho)=\int_{\mbb{H}\times\mbb{H}}\rho(z)G^{\mrm{Fr}}_{\mbb{H}^{\eta}_{t}}(z,w)\rho(w)d(\mu\otimes\mu)(z,w)	
\end{equation*}
is non-increasing in the time variable $t\in [0,T]$.
This implies that $(\mfrak{h}_{t},\rho)$, $t\in [0,T]$, is a Brownian motion such that we can regard $-E^{\mrm{Fr}}_{t}(\rho)$ as a time variable.
Thus $(\mfrak{h}_{T},\rho)$ is normally distributed with mean $(\mfrak{h}_{0},\rho)$ and variance $-E^{\mrm{Fr}}_{T}(\rho)-(-E^{\mrm{Fr}}_{0}(\rho))$.
The random variable $(H^{\mrm{Fr}}_{\mbb{H}}\circ f^{T}_{T},\rho)$ is also normally distributed with mean zero and variance $E^{\mrm{Fr}}_{T}(\rho)$
by the conformal invariance of the GFF.
Since the random variable $(H^{\mrm{Fr}}_{\mbb{H}}\circ f^{T}_{T},\rho)$ is conditionally independent of $(\mfrak{h}_{T},\rho)$,
their sum $(\mfrak{p}_{T},\rho)$ is a normal random variable with mean $(\mfrak{h}_{0},\rho)$ and variance $E^{\mrm{Fr}}_{0}(\rho)$
coinciding with $(\mfrak{h}_{0}+H^{\mrm{Fr}}_{\mbb{H}},\rho)=(\mfrak{p}_{0},\rho)$ in probability law.
This implies $\mfrak{p}_{T}\overset{\mrm{(law)}}{=}\mfrak{p}_{0}$ as $C_{\nabla}^{\infty}(\mbb{H})^{\prime}$-valued random fields.
The proof of Proposition \ref{prop:stationarity_welding} is complete.

\section{Proof of Theorem \ref{thm:answer_flowline}}
\label{sect:coupling_forward}
This section is devoted to a proof of another main result Theorem \ref{thm:answer_flowline} in the present paper.
Again, we let $\{\bm{X}_{t}=(X_{t}^{(1)},\dots, X_{t}^{(N)}):t\ge 0\}$ be a set of driving processes satisfying (\ref{eq:driving_process_Schramm})
associated with parameter $\kappa>0$ and functions $\{F^{(i)}(\bm{x})\}_{i=1}^{N}$
with initial conditions $(X^{(1)}_{0},\dots, X^{(N)}_{0})=(x_{1},\dots, x_{N})\in\mrm{Conf}^{<}_{N}(\mbb{R})$.
We also assume that the Loewner equation (\ref{eq:multiple_Loewner}) driven by $\{\bm{X}_{t};t\ge 0\}$ has the unique solution $g_{t}$, $t\ge 0$
and determines $N$ non-intersecting slits $\{\eta^{(i)}\}_{i=1}^{N}$ in $\mbb{H}$ anchored on $\mbb{R}$.

\subsection{Key statement}
At each $z\in\mbb{H}$, consider a stochastic process
\begin{equation*}
	\mfrak{n}^{\mfrak{I}}_{t}(z)=\frac{-2}{\sqrt{\kappa}}\sum_{i=1}^{N}\arg (z-X_{t}^{(i)}),\ \ \ t\ge 0.
\end{equation*}
At each time $t\in [0,\infty)$, the random field $\mfrak{n}^{\mfrak{I}}_{t}$ is harmonic on $\mbb{H}$ and its boundary value changes at $N$ points $X^{(i)}_{t}\in \mbb{R}$, $i=1,\dots, N$.
Let $H^{\mrm{Dir}}_{\mbb{H}}$ be independent of $\{B^{(i)}_{t}:t\ge 0\}_{i=1}^{N}$.
Put $H^{\bm{X}_{t},\beta,\mfrak{I}}_{\mbb{H}}=\mfrak{n}^{\mfrak{I}}_{t}+H^{\mrm{Dir}}_{\mbb{H}}$, 
where $(\beta_{1},\dots, \beta_{N})=(\frac{2}{\sqrt{\kappa}},\dots,\frac{2}{\sqrt{\kappa}})$ are fixed here and in the sequel.
We also define
\begin{equation*}
	\mfrak{h}^{\mfrak{I}}_{t}(z):=\mfrak{n}^{\mfrak{I}}_{t}(g_{t}(z))-\chi \arg\pr{g}_{t}(z)
\end{equation*}
with $\chi\in\mbb{R}$ and
\begin{equation*}
	\mfrak{p}^{\mfrak{I}}_{t}:=H^{\bm{X}_{t},\beta,\mfrak{I}}_{\mbb{H}}\circ g_{t}-\chi \arg\pr{g}_{t}=\mfrak{h}^{\mfrak{I}}_{t}+H_{\mbb{H}}^{\mrm{Dir}}\circ g_{t}.
\end{equation*}
Note that $(\mbb{H}^{\eta}_{t},H^{\bm{X}_{t},\beta,\mfrak{I}}_{\mbb{H}}\circ g_{t}-\chi \arg\pr{g}_{t})\sim_{\chi}(\mbb{H},H^{\bm{X}_{t},\beta,\mfrak{I}}_{\mbb{H}})$.
Due to the initial condition $g_{0}(z)=z\in\mbb{H}$, we can see that $\mfrak{p}^{\mfrak{I}}_{0}=H^{\bm{x},\beta,\mfrak{I}}_{\mbb{H}}$,
where $\bm{x}=(x_{1},\dots, x_{N})\in\mrm{Conf}^{<}_{N}(\mbb{R})$ are the initial values of the driving processes.

\begin{prop}
\label{prop:stationarity_flowline}
Let $\kappa\in (0,4]$, $N\in\mbb{Z}_{\ge 1}$ and $\chi\in\mbb{R}$.
Suppose that $\{\bm{X}_{t}=(X^{(1)}_{t},\dots, X^{(N)}_{t}):t \ge 0\}$ is the solution of the system of SDEs (\ref{eq:driving_process_Schramm})
starting at $\bm{x}=(x_{1},\dots, x_{N})\in\mrm{Conf}^{<}_{N}(\mbb{R})$.
If $\chi=\frac{2}{\sqrt{\kappa}}-\frac{\sqrt{\kappa}}{2}$ and the functions $\{F^{(i)}(\bm{x})\}_{i=1}^{N}$ are given by (\ref{eq:function_canonical}),
then at each time $T\in [0,\infty)$, two $C^{\infty}_{0}(\mbb{H})^{\prime}$-valued random fields 
$\mfrak{p}^{\mfrak{I}}_{0}$ and $\mfrak{p}^{\mfrak{I}}_{T}$ obey the same probability law.
\end{prop}

\subsection{Proof of Proposition \ref{prop:stationarity_flowline}}
The proof is very similar to that of Proposition \ref{prop:stationarity_welding}.
Thus we omit the computational details.
The following lemmas play key roles:

\begin{lem}
At each $z\in \mbb{H}$,
the stochastic process $\mfrak{h}^{\mfrak{I}}_{t}(z)$, $t\ge 0$ is a local martingale with increment
\begin{equation*}
	d\mfrak{h}^{\mfrak{I}}_{t}(z)=\sum_{i=1}^{N}\mrm{Im}\frac{2}{g_{t}(z)-X_{t}^{(i)}}dB_{t}^{(i)},\ \ t\ge 0,
\end{equation*}
if the functions $\{F^{(i)}(\bm{x})\}_{i=1}^{N}$ are chosen as (\ref{eq:function_canonical}).
\end{lem}

Thus, at each $t\in [0,\infty)$, the cross variation between $\mfrak{h}^{\mfrak{I}}_{t}(z)$ and $\mfrak{h}^{\mfrak{I}}_{t}(w)$ is given by
\begin{equation*}
	d\braket{\mfrak{h}^{\mfrak{I}}(z),\mfrak{h}^{\mfrak{I}}(w)}_{t}=\sum_{i=1}^{N}\left(\mrm{Im}\frac{2}{g_{t}(z)-X^{(i)}_{t}}\right)\left(\mrm{Im}\frac{2}{g_{t}(w)-X^{(i)}_{t}}\right)dt,\ \ z,w\in\mbb{H},
\end{equation*}
which turns out to be expressed using the Green function.

\begin{lem}
Let $G^{\mrm{Dir}}_{\mbb{H}^{\eta}_{t}}(z,w):=G^{\mrm{Dir}}_{\mbb{H}}(g_{t}(z),g_{t}(w))$, $t\ge 0$, $z,w\in\mbb{H}^{\eta}_{t}$.
Then we have
\begin{equation*}
	d\braket{\mfrak{h}^{\mfrak{I}}(z),\mfrak{h}^{\mfrak{I}}(w)}_{t}=-dG^{\mrm{Dir}}_{\mbb{H}^{\eta}_{t}}(z,w),\ \ t\ge 0,\ \ z,w\in\mbb{H}^{\eta}_{t}.
\end{equation*}
\end{lem}

For $\rho\in C^{\infty}_{0}(\mbb{H})$, the quadratic variation of $(\mfrak{h}^{\mfrak{I}}_{t},\rho)$ becomes
\begin{equation*}
	d\braket{(\mfrak{h}^{\mfrak{I}},\rho), (\mfrak{h}^{\mfrak{I}},\rho)}=-dE^{\mrm{Dir}}_{t}(\rho),	
\end{equation*}
where
\begin{equation*}
	E^{\mrm{Dir}}_{t}(\rho)=\int_{\mbb{H}^{\eta}_{t}\times\mbb{H}^{\eta}_{t}}\rho(z)G_{\mbb{H}^{\eta}_{t}}(z,w)\rho(w)d(\mu\otimes\mu)(z,w)	
\end{equation*}
is the Dirichlet energy of $\rho$ in the domain $\mbb{H}^{\eta}_{t}$, $t\ge 0$.
Here $\rho$ is supposed to be a function on $\mbb{H}^{\eta}_{t}$ by restriction.
Since the process $\mbb{H}^{\eta}_{t}$, $t\ge 0$ is non-increasing, regarding $-E^{\mrm{Dir}}_{t}(\rho)$ as a new time variable,
the stochastic process $(\mfrak{h}^{\mfrak{I}}_{t},\rho)$ is a Brownian motion.
This implies that at any fixed time $T\in [0,\infty)$,
the random variable $(\mfrak{h}^{\mfrak{I}}_{T},\rho)$ is normal with mean $(\mfrak{h}^{\mfrak{I}}_{0},\rho)$ and variance $-E^{\mrm{Dir}}_{T}(\rho)-(-E^{\mrm{Dir}}_{0}(\rho))$.
The random variable $(H^{\mrm{Dir}}_{\mbb{H}}\circ g_{T},\rho)$ is also a mean-zero normal variable with variance $E^{\mrm{Dir}}_{T}(\rho)$.
Since $(\mfrak{h}^{\mfrak{I}}_{T},\rho)$ and $(H^{\mrm{Dir}}_{\mbb{H}}\circ g_{T},\rho)$ are conditionally independent,
their sum $(\mfrak{p}^{\mfrak{I}}_{T},\rho)$ is normally distributed with mean $(\mfrak{h}^{\mfrak{I}}_{0},\rho)$ and variance $E^{\mrm{Dir}}_{0}(\rho)$,
thus it coincides with $(\mfrak{p}^{\mfrak{I}}_{0},\rho)$ in probability law.

\subsection{Arguments}
\label{subsect:arguments_flow_lines}
\begin{figure}[h]
\centering
\includegraphics[width=\hsize]{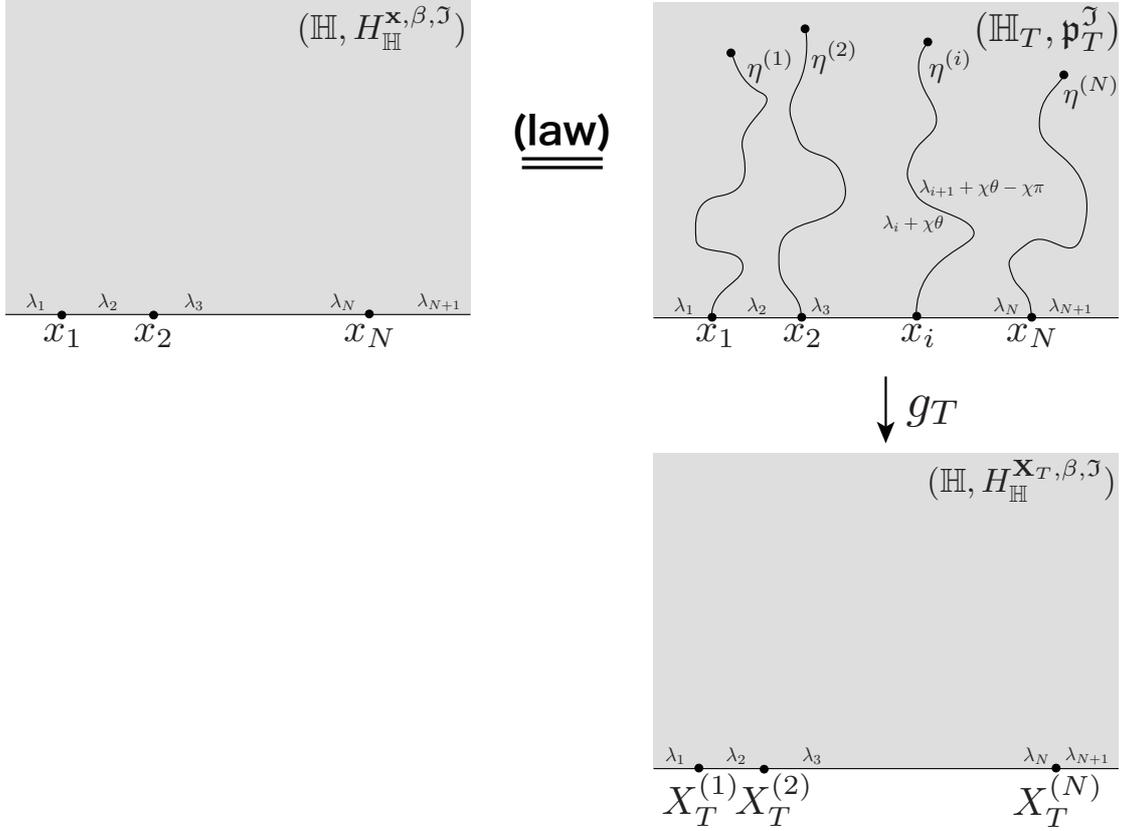}
\caption{Geometric interpretation of Proposition \ref{prop:stationarity_flowline}: The boundary values on the real axis are defined by $\lambda_{i}=-2\pi(N-i-1)/\sqrt{\kappa}$, $i=1,\dots, N+1$.}
\label{fig:coupling_forward}
\end{figure}
We here present a geometric interpretation of Proposition \ref{prop:stationarity_flowline} and see that it indeed proves Theorem \ref{thm:answer_flowline}.
That is, Proposition \ref{prop:stationarity_flowline} provides two distinct samplings of $C^{\infty}_{0}(\mbb{H})^{\prime}$-valued random fields
which obey the same probability law.
The first one is to directly sample the random field $\mfrak{p}^{\mfrak{I}}_{0}=H^{\bm{x},\beta,\mfrak{I}}_{\mbb{H}}$ (see the upper-left picture in Fig.~\ref{fig:coupling_forward}),
and the other one is first to sample multiple SLE paths $\{\eta^{(i)}\}_{i=1}^{N}$
up to time $T\in (0,\infty)$, sample the random field $\mfrak{p}^{\mfrak{I}}_{T}=\mfrak{h}^{\mfrak{I}}_{T}+H^{\mrm{Dir}}_{\mbb{H}}\circ g_{T}$
on the domain $\mbb{H}^{\eta}_{T}$, and then extend it to $\mbb{H}$ (see the upper-right picture of Fig.~\ref{fig:coupling_forward}).
Coincidence in probability law between these two samplings roughly means that there is a one-to-one correspondence
among instances of two samplings with the same weights.
In particular, with each instance $h$ of $\mfrak{p}^{\mfrak{I}}_{0}$, one can associate an $N$-tuple of multiple SLE paths $\{\eta^{(i)}\}_{i=1}^{N}$.
We describe this correspondence more concretely in the sequel.

Notice that $\mfrak{n}^{\mfrak{I}}_{t}$, $t\ge 0$, is the unique harmonic function with boundary conditions
\begin{equation*}
	\mfrak{n}^{\mfrak{I}}_{t}(x)=-\frac{2\pi}{\sqrt{\kappa}}(N-i),\ \ \mbox{if}\ x\in (X_{t}^{(i)},X_{t}^{(i+1)}),\ \ i=0,1,\dots, N.
\end{equation*}
Here we follow the convention that $X_{t}^{(0)}=-\infty$ and $X_{t}^{(N+1)}=+\infty$.
In particular, it has discontinuity at $X_{t}^{(i)}$ by $2\pi/\sqrt{\kappa}$ along $\mbb{R}$, $i=1,\dots, N$.
Note that $H^{\bm{X}_{t},\beta,\mfrak{I}}_{\mbb{H}}$, $t\ge 0$ is regarded as the GFF with the same boundary condition as $\mfrak{n}^{\mfrak{I}}_{t}$, $t\ge 0$.

\begin{figure}[h]
\centering
\includegraphics[width=.5\hsize]{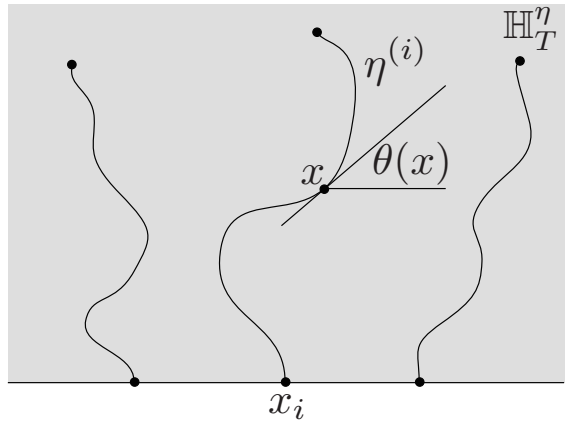}
\caption{The angle $\theta(x)$}
\label{fig:angle}
\end{figure}

Let us investigate the behavior of $\mfrak{p}^{\mfrak{I}}_{T}$, $T\in [0,\infty)$ near the boundary of $\mbb{H}^{\eta}_{T}$.
Near a point $x\in (x_{i},x_{i+1})$ on the real axis, we have $\lim_{z\to x+0\sqrt{-1}}\arg \pr{g}_{T}(z)= 0$.
Thus $\mfrak{p}^{\mfrak{I}}_{T}$, $T\in [0,\infty)$ has the same boundary value as $\mfrak{p}^{\mfrak{I}}_{0}$ on the real axis.
We next suppose that $x$ lies on the $i$-th strand of the $N$-tuple of SLE paths, i.e., $x \in \eta^{(i)}(0,T)$.
Although the strand $\eta^{(i)}(0,T)$ is not a smooth curve, we temporarily let $\theta(x)$ be
the angle of the tangent line of $\eta^{(i)}(0,T)$ at $x$ (see Fig.~\ref{fig:angle}).
Then we have $\lim_{z\to x +}\arg\pr{g}_{T}(z)=\pi-\theta(x)$ where $z$ approaches $x$ from the right,
while $\lim_{z\to x-}\arg\pr{g}_{T}(z)=-\theta(x)$ where $z$ approaches $x$ from the left.
Thus
\begin{align*}
	\lim_{z\to x+}\mfrak{p}^{\mfrak{I}}_{T}(z)&=-\frac{2\pi}{\sqrt{\kappa}}(N-i)+\chi\theta(x)-\chi\pi, \\
	\lim_{z\to x-}\mfrak{p}^{\mfrak{I}}_{T}(z)&=-\frac{2\pi}{\sqrt{\kappa}}(N-i+1)+\chi\theta(x)
\end{align*}
(see the upper- and lower-right pictures of Fig.~\ref{fig:coupling_forward}).
Although these two values themselves are not well-defined because $\theta(x)$ is not,
their difference does not depend on $\theta(x)$ yielding
\begin{equation*}
	\lim_{z\to x+}\mfrak{p}^{\mfrak{I}}_{T}(z)-\lim_{z\to x-}\mfrak{p}^{\mfrak{I}}_{T}(z)=\frac{\sqrt{\kappa}\pi}{2},
\end{equation*}
where we have used the relation $\chi=\frac{2}{\sqrt{\kappa}}-\frac{\sqrt{\kappa}}{2}$.
Thus we will see that $\mfrak{p}^{\mfrak{I}}_{T}$ has discontinuity by $\sqrt{\kappa}\pi/2$
from the left side to the right side of a strand.
Conversely, from Proposition \ref{prop:stationarity_flowline}, one can find in an instance $h$ of $\mfrak{p}^{\mfrak{I}}_{0}=H^{\bm{x},\beta,\mfrak{I}}_{\mbb{H}}$
strands evolving from $x_{i}$, $i=1,\dots, N$
so that the value of $h$ has discontinuity by $\sqrt{\kappa}\pi/2$ from the left side to the right side of a strand.
Following the argument in ~\cite{SchrammSheffield2013,MillerSheffield2016a,Sheffield2016},
these strands can be regarded as the flow lines of the vector field $e^{\sqrt{-1}H^{\bm{x},\beta,\mfrak{I}}_{\mbb{H}}/\chi}$ starting from $x_{i}$, $i=1,\dots, N$.
Moreover, the law of these strands agrees with the one of the slits $\{\eta^{(i)}\}_{i=1}^{N}$ determined by the multiple SLE$_{\kappa}$.
Therefore, Theorem \ref{thm:answer_flowline} is proved.

\section{Concluding remarks}
\label{sect:concluding_remarks}
As conclusion, we make some discussions and remarks on related topics and future directions.
\subsection{Pathwise uniqueness of coupling}
In the present paper, we found couplings between GFFs and multiple SLEs in suitable senses. For both of the conformal welding problem and the flow line problem, we heuristically argued in Subsects. \ref{subsect:conformal_welding_problem} and \ref{subsect:arguments_flow_lines} that multiple SLEs are uniquely determined by the GFFs. In ~\cite{Dubedat2009,MillerSheffield2016a}, the authors rigorously proved this kind of pathwise uniqueness for a single curve with assistance of the dual SLE. It is an interesting and important future direction to consider how to generalize their arguments to our case, where we treat multiple curves at once.

\subsection{Dynamics of the Dyson model}
\label{subsec:dyson_model}
The conformal welding problem introduced in Section \ref{subsec:conformal_welding} requires precise description of the
statistical behavior of slits $\{\eta^{(i)}\}_{i=1}^N$ in $\mbb{H}$, which are non-intersecting.
Our strategy to solve this problem is to identify $\{\eta^{(i)}\}_{i=1}^N$ in $\mbb{H}$ with ``$N$-tuple of SLE curves."
Based on Theorem \ref{thm:multple_Loewner} from ~\cite{RothSchleissinger2017} for the deterministic multiple Loewner equation (\ref{eq:multiple_Loewner}),
we have reduced the problem for $\{\eta^{(i)}\}_{i=1}^N$ in $\mbb{H}$ to the problem to find a stochastic process
$\bm{X}_{t}=(X^{(1)}_{t}, \dots, X^{(N)}_{t}) \in \mathbb{R}^N$, $t \geq 0$, of $N$ particles on $\mathbb{R}$.
Then we have assumed that $\bm{X}_t$, $t \geq 0$ satisfies a system of SDEs in a general form (\ref{eq:driving_process_Schramm}). 
There, the drift terms $\int_0^t F^{(i)}(\bm{X}_s) ds$, $i=1, \dots, N$, $t \geq 0$, which determine the interaction among $N$ particles, are arbitrary. 

In the formulation of Problem \ref{prob:1}, we have assumed that the slits $\{\eta^{(i)}\}_{i=1}^{N}$ are non-intersecting,
which implies that the stochastic process $\bm{X}_{t}$, $t\ge 0$ must be non-colliding.
The main result in the present paper given by Theorem \ref{thm:answer_welding} states that the solution can be given by a proper time
change of the Dyson model, which may be the most studied process in non-colliding particle systems in probability theory and random matrix theory
(see, for instance ~\cite{Forrester2010, AndersonGuionnetZeitouni2010,Katori2015}).
As shown in (\ref{eq:Dyson_model}), in the Dyson model, the repulsive force acts between any pair of particles, whose strength is proportional
to the inverse of distances between the particles and the proportionality coefficient is given by $\beta/2,$ $\beta \in (0, \infty)$.
We note that the $\gamma$-QS-$(N+1)$-MBPs of the $(\bm{X}, \alpha)$- standard type has the set of parameters
$\gamma \in (0,2)$ and $\alpha_i$, $i=1, \dots, N$,
and the multiple SLE does one parameter $\kappa \in (0, \infty)$.
Theorem \ref{thm:answer_welding} determines the relations among them as
\begin{equation*}
\kappa = \gamma^2, \quad
\alpha_i=\frac{2}{\gamma} \quad (i=1, \dots, N), \quad
\beta=\frac{8}{\kappa}.
\end{equation*}
The equality $\kappa=\gamma^2$ is the same as that given in \cite{Sheffield2016} and $\alpha_i=2/\gamma$, $i=1, \dots, N$
are a simple $N$-variable extension of his result $\alpha_1=2/\gamma$ for the original conformal welding problem with two marked boundary points.
The equality $\beta=8/\kappa$ is found in the literatures \cite{Cardy2003a,Cardy2003b,Cardy2004,BauerBernardKytola2005},
but its derivations heavily depended on CFT and the so-called group theoretical formulation of SLE \cite{BauerBernard2003,BauerBernard2004,SK2018} (see Appendix \ref{app:driving_processes}).
Our derivation given in the proof of Lemma \ref{lem:key_martingale} is purely probability theoretical and simple.
Since $\gamma\in (0,2)$ for the LQG, we have $\beta> 2$.
Therefore, the resulting time change of the Dyson model indeed is non-colliding.

\begin{figure}[t]
\centering
\includegraphics[width=\hsize]{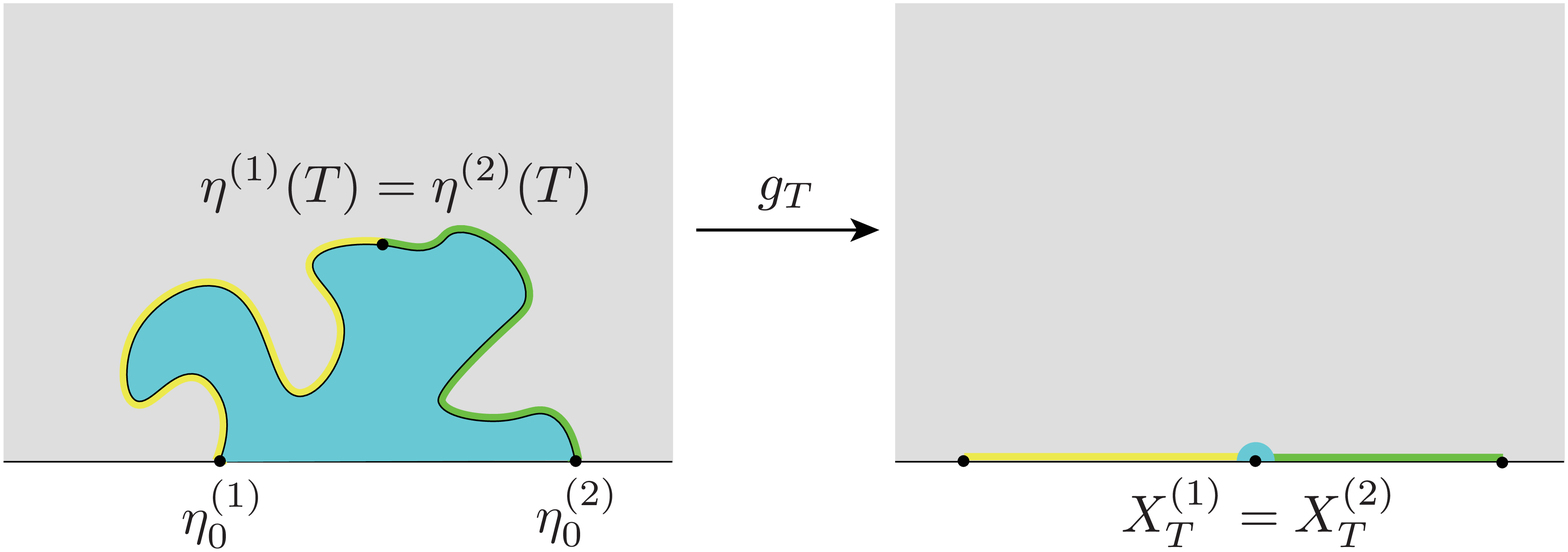}
\caption{(I) Tips collide}
\label{fig:tips_collide}
\vspace{1cm}
\includegraphics[width=\hsize]{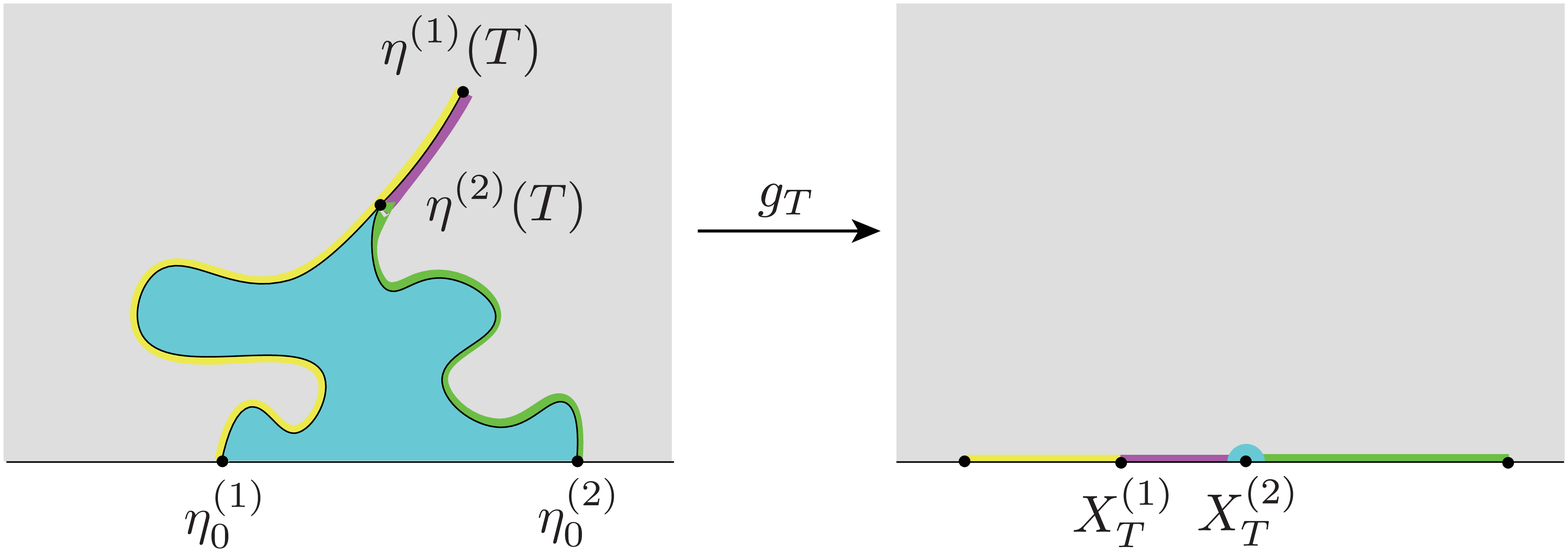}
\caption{(II) A tip collides with an already existing slit}
\label{fig:tip_collides_slit}
\end{figure}

Note that the non-colliding condition $\beta\ge 1$ for the Dyson model \cite{RogersShi1993,CepaLepingle1997,GraczykMalecki2013,GraczykMalecki2014} corresponds to $\kappa\in (0,8]$ for the multiple SLE,
while when $\kappa > 4$ the multiple SLE curves will collide with each other and become self-intersecting in $\overline{\mathbb{H}}$.
In the region $\kappa\in (4, 8]$, the correspondence between slits and driving processes will be explained as follows.
If two slits collide with each other, this event is classified into two cases.
(I) Two tips of slits collide with each other (see Fig.~\ref{fig:tips_collide}).
(II) A tip collides with an already existing slit (see Fig.~\ref{fig:tip_collides_slit}).
Since each of the driving processes is the image of a tip of a slit under the uniformization map,
two of the driving processes collide with each other when the event (I) occurs.
On the other hand, when the event (II) occurs, the driving processes are non-colliding even though the corresponding SLE slits are colliding.
Following this argument, we could expect that the Dyson model will fall into three classes.
When $\beta\ge 2$, the particles are non-colliding and the corresponding SLE slits are non-intersecting.
When $\beta\in [1,2)$, the particles are non-colliding, but the event (II) almost surely occurs.
When $\beta \in (0,1)$, the particles collide, and correspondingly, the event (I) almost surely occurs.
Though it is known \cite{RogersShi1993,CepaLepingle1997,GraczykMalecki2013,GraczykMalecki2014} that the colliding/non-colliding transition occurs at $\beta=1$,
the possible phenomenon that the characteristics of the Dyson model changes at $\beta=2$ has not been well studied so far.
It would be an interesting future direction to find a property that distinguishes the Dyson model of $\beta\in [1,2)$ and that of $\beta\ge 2$.
We will also have to make discussions analogous to that in ~\cite{RohdeSchramm2005} to settle this classification.

\subsection{Other driving processes}
As interacting particle systems related with random matrix theory,
a variety of non-colliding particle systems have been studied (see, for instance, ~\cite{KatoriTanemura2004}).
We hope that we can address the conformal welding problems in {\it other situations}
and, in solving them, interesting relations between non-colliding particle systems and multiple SLEs will be discovered.
We will depict examples of such {\it other situations}
for which the conformal welding problem is solvable and produces an another type of driving processes.

\subsubsection{Inhomogeneous systems}
The setting (\ref{eq:reverse_multiple_SLE}) and (\ref{eq:driving_process_reverse}) for the conformal welding problem can be generalized as follows:
\begin{align}
\label{eq:inhomo_Loewner}
	\frac{d}{dt}f^{T}_{t}(z)&=-\sum_{i=1}^{N}\frac{2\lambda_{i}}{f^{T}_{t}(z)-Y^{(i)}_{T;t}},\ \ t\in [0,T],\ \ f^{T}_{0}(z)=z\in\mbb{H},\\
\label{eq:inhomo_driving}
	dY^{(i)}_{T;t}&=\sqrt{\kappa_{i}}dB_{t}^{(i)}-F^{(i)}(\bm{Y}_{T;t})dt,\ \ t\in [0,T],\ \ i=1,\dots, N,
\end{align}
where $\lambda_{i}>0$, $\kappa_{i}>0$, $i=1,\dots, N$ with $\sum_{i=1}^{N}\lambda_{i}=N$ ~\cite{Schleissinger2013, RothSchleissinger2017,delMonacoHottaSchleissinger2018}.
Let
\begin{equation*}
	\mfrak{h}_{t}^{\ast}(z)=\sum_{i=1}^{N}\alpha_{i}\log (f^{T}_{t}(z)-Y^{(i)}_{T;t})+Q\log f^{T\prime}_{t}(z),\ \ z\in \mbb{H},\ \ t\in [0,T],
\end{equation*}
where $\alpha_{i}$, $i=1,\dots, N$ are indeterminate real numbers and $Q=\frac{2}{\gamma}+\frac{\gamma}{2}$.
By the similar calculation to that given in the proof of Lemma \ref{lem:key_martingale}, we can show that
\begin{align*}
	d\mfrak{h}^{\ast}_{t}(z)=
	&-\sum_{i=1}^{N}\frac{\alpha_{i}\kappa_{i}}{f^{T}_{t}(z)-Y^{(i)}_{T;t}}dB^{(i)}_{t}+\sum_{i=1}^{N}\frac{2C(\alpha_{i},\kappa_{i},\lambda_{i},\gamma)}{(f^{T}_{t}(z)-Y^{(i)}_{T;t})^{2}}dt\\
	&+\sum_{i=1}^{N}\frac{1}{f^{T}_{t}(z)-Y^{(i)}_{T;t}}\left(\alpha_{i}F^{(i)}(\bm{Y}_{T;t})-2\sum_{\substack{j=1\\j\neq i}}^{N}\frac{\alpha_{i}\lambda_{j}+\alpha_{j}\lambda_{i}}{Y^{(i)}_{T;t}-Y^{(j)}_{T;t}}\right)dt,\ \ t\in [0,T],
\end{align*}
with
\begin{equation}
\label{eq:inhomo_drift_const}
	C(\alpha_{i},\kappa_{i},\lambda_{i},\gamma)=-\left(\lambda_{i}+\frac{\kappa_{i}}{4}\right)\alpha_{i}+Q\lambda_{i},\ \ i=1,\dots, N.
\end{equation}
Hence if
\begin{equation}
\label{eq:inhomo_martingale_const}
	C(\alpha_{i},\kappa_{i},\lambda_{i},\gamma)=0,\ \ i=1,\dots, N
\end{equation}
and
\begin{align}
\label{eq:inhomo_martingale_function}
	F^{(i)}(\bm{x})=
	&\frac{2}{\alpha_{i}}\sum_{\substack{j=1\\ j\neq i}}^{N}\frac{\alpha_{i}\lambda_{j}+\alpha_{j}\lambda_{i}}{x_{i}-x_{j}}\\
	=&2\sum_{\substack{j=1\\ j\neq i}}^{N}\frac{\lambda_{j}}{x_{i}-x_{j}}+\frac{2\lambda_{i}}{\alpha_{i}}\sum_{\substack{j=1\\ j\neq i}}^{N}\frac{\alpha_{j}}{x_{i}-x_{j}},\ \ i=1,\dots, N, \notag
\end{align}
then $\mfrak{h}^{\ast}_{t}(z)$, $z\in\mbb{H}$, $t\in [0,T]$ becomes a local martingale.
When $\lambda_{i}=1$, $\kappa_{i}= \kappa$, $i=1,\dots, N$,
(\ref{eq:inhomo_martingale_const}) with (\ref{eq:inhomo_drift_const}) gives
\begin{equation*}
	\alpha_{i}=\frac{4Q}{4+\kappa}=\frac{2Q}{\sqrt{\kappa}(\frac{2}{\sqrt{\kappa}}+\frac{\sqrt{\kappa}}{2})},\ \ i=1,\dots, N.
\end{equation*}
That is, the weights to the marked boundary points are homogeneous.
When we further assume the relation $\kappa=\gamma^{2}$,
then $\alpha_{i}=\frac{2}{\gamma}$ as we have seen in one of our main theorems (Theorem \ref{thm:answer_welding}).

Inhomogeneous setting of the conformal welding problem with $\alpha_{i}\neq\alpha_{j}$, $i\neq j$ in general
(as well as the inhomogeneous flow line problem with $\beta_{i}\neq \beta_{j}$, $i\neq j$ in general)
will be studied in which inhomogeneous multiple SLE (\ref{eq:inhomo_Loewner})
driven by inhomogeneous interacting particles on $\mbb{R}$ (\ref{eq:inhomo_driving}) shall be analyzed
under the conditions (\ref{eq:inhomo_martingale_const}) with (\ref{eq:inhomo_drift_const}), and (\ref{eq:inhomo_martingale_function})
to solve the problems.

\subsubsection{Multiple quadrant SLE and the Wishart process}
In this paper, we have formulated the conformal welding problem for a $\gamma$-QS-$(N+1)$-MBPs
$[\mbb{H},H^{\bm{X},\alpha}_{\mbb{H}},(\bm{X},\infty)]_{\gamma}$ of the $(\bm{X},\alpha)$-standard type.
In solving this, we have adopted the form of multiple Loewner equation (\ref{eq:multiple_Loewner})
and assumed that the set of driving processes is determined by the system of SDEs (\ref{eq:driving_process_Schramm})
with drift functions $\{F^{(i)}(\bm{x})\}_{i=1}^{N}$
motivated by preceding works \cite{BauerBernardKytola2005,RothSchleissinger2017}.
As a result, we have found that the parameters $\kappa=\gamma^{2}$ and $\alpha_{i}=\frac{2}{\gamma}$ are determined
and the functions are chosen as (\ref{eq:function_canonical})
to obtain a one-parameter family $\{[\mbb{H},H^{\bm{X}_{T},\alpha}_{\mbb{H}},(\bm{X}_{T},\infty)]_{\gamma}:T\in (0,\infty)\}$ 
of $\gamma$-QSs-$(N+1)$-MBPs,
for each of which the conformal welding problem is solvable.

Notice that there is room for changing the model of uniformization maps.
As a generalized multiple Loewner equation for $N$ slits, we consider the following form
\begin{equation}
	\label{eq:multiple_Loewner_general}
	\frac{d}{dt}g_{t}(z)=\Psi(g_{t}(z),\bm{X}_{t}),\ \ t\ge 0,\ \ g_{0}(z)=z,
\end{equation}
where $\Psi(z,\bm{x})$ is a suitable functions of $z$ and $\bm{x}=(x_{1},\dots, x_{N})$,
and $\{\bm{X}_{t}=(X_{t}^{(1)},\dots, X_{t}^{(N)}):t\ge 0\}$ is a set of driving processes.
We do not specify the domain of definition for the function $\Psi(z,\bm{x})$ since it will depend on models.
When we take
\begin{equation*}
	\Psi(z,\bm{x})=\sum_{i=1}^{N}\frac{2}{z-x_{i}},\ \ z\in\mbb{H},\ \ \bm{x}\in\mbb{R}^{N},
\end{equation*}
the associated Loewner equation (\ref{eq:multiple_Loewner_general}) reduces to (\ref{eq:multiple_Loewner}).

Let us see the case for another choice of $\Psi(z,\bm{x})$.
Let $\mbb{O}:=\{z\in\mbb{C}|\mrm{Re}z>0,\ \mrm{Im}z>0\}$ be an orthant in $\mbb{C}$.
We adopt
\begin{equation*}
	\Psi(z,\bm{x})=\Psi_{\mbb{O}}(z,\bm{x}):=\sum_{i=1}^{N}\left(\frac{2}{z-x_{i}}+\frac{2}{z+x_{i}}\right)+\frac{4\delta}{z},\ \ z\in\mbb{O},\ \ \bm{x}\in (\mbb{R}_{>0})^{N}.
\end{equation*}
Here $\delta\in\mbb{R}$ is a parameter and $\mbb{R}_{>0}=\{x\in\mbb{R}|x>0\}$ is the set of positive real numbers.
The associated Loewner equation (\ref{eq:multiple_Loewner_general}) becomes
\begin{align}
\label{eq:quadrant_multiple_SLE}
	\frac{d}{dt}g_{t}(z)&=\sum_{i=1}^{N}\left(\frac{2}{g_{t}(z)-X^{(i)}_{t}}+\frac{2}{g_{t}(z)+X^{(i)}_{t}}\right)+\frac{4\delta}{g_{t}(z)},\ \ t\ge 0,\\
	g_{0}(z)&=z\in\mbb{O}. \notag
\end{align}
We again assume that the set of driving processes $\{\bm{X}_{t}=(X^{(1)}_{t},\dots, X^{(N)}_{t})\in (\mbb{R}_{>0})^{N}:t\ge 0\}$ solves the sytem of SDEs
\begin{equation*}
	dX^{(i)}_{t}=\sqrt{\kappa}dB^{(i)}_{t}+F^{(i)}(\bm{X}_{t})dt,\ \ t\ge 0,\ \ i=1,\dots, N,
\end{equation*}
where $\kappa>0$ is a parameter, 
$\{B^{(i)}_{t}:t\ge 0\}_{i=1}^{N}$ are mutually independent standard Brownian motions
and $\{F^{(i)}(\bm{x})\}_{i=1}^{N}$ are suitable functions of $\bm{x}=(x_{1},\dots, x_{N})$ 
so that $\bm{X}_{t}$ lies in $(\mbb{R}_{>0})^{N}$ for all $t\ge 0$.
The equation (\ref{eq:quadrant_multiple_SLE}) is the multiple version of the {\it quadrant Loewner equation}
considered in ~\cite{Takebe2014}.
We assume that, if the initial value of $\bm{X}_{t}$ satisfies $0<X^{(1)}_{0}<X^{(2)}_{0}<\cdots <X^{(N)}_{0}$,
each realization determines $N$ non-colliding and non-intersecting slits $\{\eta^{(i)}:(0,\infty)\to\mbb{O}\}_{i=1}^{N}$
anchored on $\mbb{R}_{>0}$: $\eta^{(i)}_{0}=X^{(i)}_{0}$, $i=1,\dots, N$,
i.e., $g_{t}(\cdot)$, $t\ge 0$, becomes a uniformization map
\begin{equation*}
	g_{t}:\mbb{O}^{\eta}_{t}:=\mbb{O}\Big\backslash \bigcup_{i=1}^{N}\eta^{(i)}(0,t]\to\mbb{O}.
\end{equation*}

Let $\gamma\in (0,2)$.
For $N$ points $\bm{x}=(x_{1},\dots, x_{N})$, where $0<x_{1}<x_{2}<\cdots <x_{N}$ and
an $N$-tuple of real numbers $\alpha=(\alpha_{1},\dots, \alpha_{N})$,
we define the following function on $\mbb{O}$:
\begin{equation*}
	w^{\bm{x},\alpha}_{\mbb{O}}(z)=\sum_{i=1}^{N}\alpha_{i}\left(\log |z-x_{i}|+\log |z+x_{i}|\right)+Q\log |z|,
\end{equation*}
where $Q=\frac{2}{\gamma}+\frac{\gamma}{2}$,
and a $C_{\nabla}^{\infty}(\mbb{O})^{\prime}$-valued random field 
$H^{\bm{x},\alpha}_{\mbb{O}}:=H^{\mrm{Fr}}_{\mbb{O}}+w^{\bm{x},\alpha}_{\mbb{O}}$.
For a random $N$-point configuration $\bm{X}=(X_{1},\dots, X_{N})$ valued in $\mrm{Conf}^{<}_{N}(\mbb{R}_{>0})$,
we can see that
\begin{equation*}
	\left[\mbb{O},H^{\bm{X},\alpha}_{\mbb{O}},(\bm{X},\infty)\right]_{\gamma}
	=\left[\mbb{H},H^{\bm{X}^{2},\alpha}_{\mbb{H}},(\bm{X}^{2},\infty)\right]_{\gamma},
\end{equation*}
where $\bm{X}^{2}=((X_{1})^{2},\dots, (X_{N})^{2})$, is of the $(\bm{X}^{2},\alpha)$-standard type.

We define the {\it Wishart process} on parameters $\beta>0$ and $\nu>-1$ as a solution of the system of SDEs on $X^{\mrm{W}_{\beta,\nu}(i)}_{t}$, $t\ge 0$, $i=1,\dots, N$ such that \cite{Bru1991,KatoriTanemura2004}
\begin{align*}
	dX^{\mrm{W}_{\beta,\nu}(i)}_{t}=&dB^{(i)}_{t}+\Biggl[\frac{\beta (\nu+1)-1}{2}\frac{1}{X_{t}^{\mrm{W}_{\beta,\nu}(i)}} \\
	&+\frac{\beta}{2}\sum_{\substack{j=1\\j\neq i}}^{N}\left(\frac{1}{X^{\mrm{W}_{\beta,\nu}(i)}_{t}-X^{\mrm{W}_{\beta,\nu}(j)}_{t}}+\frac{1}{X^{\mrm{W}_{\beta,\nu}(i)}_{t}+X^{\mrm{W}_{\beta,\nu}(j)}_{t}}\right)\Biggr]dt, \\
	&t\ge 0,\ \ i=1,\dots, N.
\end{align*}

Using the function $\Psi_{\mbb{O}}(z,\bm{x})$ as a model of uniformization maps,
we can obtain solutions to the conformal welding problem.
\begin{thm}
\label{thm:quadrand_coupling}
Let $\gamma\in (0,2)$ and $N\in\mbb{Z}_{\ge 1}$.
Suppose that $\{\bm{X}_{t}=(X_{t}^{(1)},\dots, X_{t}^{(N)}):t\ge 0\}$ is a time change of the Wishart process
$\{\bm{X}^{\mrm{W}_{8/\kappa,\delta}}_{\kappa t}:t\ge 0\}$
starting at a deterministic initial state $\bm{X}_{0}=\bm{x}\in\mrm{Conf}^{<}_{N}(\mbb{R}_{>0})$.
Then, at each time $T\in (0,\infty)$, the conformal welding problem for $\left[\mbb{O}, H^{\bm{X}_{T},\alpha}_{\mbb{O}},(\bm{X}_{T},\infty)\right]_{\gamma}$
with $(\alpha_{1},\dots,\alpha_{N})=(\frac{2}{\gamma},\dots,\frac{2}{\gamma})$ is solved as follows:
\begin{enumerate}
\item 	The solution of the Loewner equation (\ref{eq:quadrant_multiple_SLE}) driven by the time change of the Wishart process
		$\{\bm{X}^{\mrm{W}_{8/\kappa,\delta}}_{\kappa t}:0\le t \le T\}$ gives a solution to Problem \ref{prob:1}.
		In other words, $g_{T}^{-1}:\mbb{O} \to \mbb{O}^{\eta}_{T}$ is the desired conformal equivalence.
\item 	The probability law for resulting slits $\{\eta^{(i)}\}_{i=1}^{N}$ is the one for the quadrant multiple SLE$_{\kappa}$.
		This gives a solution to Problem \ref{prob:2}.
\item 	Problem \ref{prob:3} is answered positively with $\bm{\eta}_{0}=\bm{x}$ a.s.
\end{enumerate}
\end{thm}

From the above observation, we could say that interacting particle systems such as the Dyson model and the Wishart process
are associated with models $\Psi(z,\bm{x})$ of uniformization maps.
Along this line, we could expect a new classification of interacting particle systems from a perspective of the Loewner theroy
and coupling with GFF.

A detail of this subject including the associated flow line problem will be published elsewhere~\cite{KatoriKoshida2020}.

\subsection{Other variants of SLE}
As we have noted above, there is room for variants of multiple SLEs in the form of (\ref{eq:multiple_Loewner_general}).
Indeed, the radial and dipole multiple SLEs are special cases of (\ref{eq:multiple_Loewner_general}).
It was shown in ~\cite{SchrammWilson2005} that, in the case of $N=1$, the chordal, radial and dipole SLEs are transformed one another by conformal mappings
in the framework of the SLE$_{\kappa;\rho}$, while the force points are allowed to be interior of the domain.
For example, the radial SLE$_{\kappa}$ is transformed to the chordal SLE$_{\kappa;\kappa-6}$ with an interior force point by a M{\"o}bius transformation.
In ~\cite{MillerSheffield2016a,Sheffield2016}, the coupling with an SLE$_{\kappa;\rho}$ and an GFF was formulated,
and our result in the present paper is also expected to be extended to the case concerning a multiple SLE$_{\kappa;\rho}$.
Then, it would be interesting to study how the transformation of these SLEs can be compatible to the coordinate transformation of
quantum surfaces in Eq.(\ref{eq:quantum_equivalence}) under the connection between the SLE and the LQG.

\subsection{The limit $N\to\infty$}
It would be interesting to consider the conformal welding problem and the flow line problem in the case with infinitely many boundary points.
In the present paper, the multiple SLE driven by an $N$-particle Dyson model arose as the solution to the conformal welding problem
for a $\gamma$-QS-$(N+1)$-MBPs of the $(\bm{X},\alpha)$-standard type.
If the method in the present paper is applicable at the limit $N\to\infty$,
it can be expected that the multiple SLE driven by an infinite dimensional Dyson model
\cite{KatoriTanemura2010, Osada2012,Osada2013,Tsai2016,OsadaTanemura2016,KawamotoOsada2018,OsadaTanemura2020} would appear in such systems.
Although the multiple SLE driven by infinitely many driving processes is not well-posed so far,
we hope that it is captured when the coupling with GFF is considered.

Another limit of $N\to\infty$ is the hydrodynamic limit of the multiple SLE \cite{delMonacoSchleissinger2016,HottaKatori2018}.
In this case, the ensemble of slits gets deterministic as $N\to\infty$.
Accordingly, a quantum surface must be subject to a boundary condition.
An interesting question is, then, how multiple slits condition the GFF on the domain
and finally impose a boundary condition.

\subsection{Discrete models converging to the present systems}
It has been reported that random planar maps converge to an SLE-decorated LQG in several topology (see ~\cite{GwynneMillerSheffield2019, HoldenSun2019, GwynneMillerSheffield2020} and references therein).
While a chordal SLE describes the scaling limit of a single interface in various critical lattice models, a multiple SLE describes
scaling limits of collections of interfaces in critical lattice models with alternating boundary conditions
(see ~\cite{BeffaraPeltolaWu2018} and references therein). 
In the present paper we introduced new kinds of continuous systems,
the {\it $\gamma$-quantum surface with $N+1$ marked boundary points} ($\gamma$-QS-$(N+1)$-MBPs) and
the {\it $\chi$-imaginary surface with $N+1$ boundary condition changing points} ($\chi$-IS-$(N+1)$-BCCPs) for $N \in \mbb{Z}_{\ge 1}$.
Both have been related with the multiple SLE driven by a Dyson model in solving the conformal welding problem and the flow line problem.
Discrete counterparts of these random systems and corresponding problems will be studied.

\appendix
\section{Construction of $\mcal{S}_{\gamma}$ and $\mcal{S}_{\gamma,N+1}$}
\label{app:construction}
In this appendix, we construct the spaces $\mcal{S}_{\gamma}$ and $\mcal{S}_{\gamma,N+1}$ as orbifolds and study their structures.

\subsection{Without marked boundary points}
Let us begin with $\mcal{S}_{\gamma}$, $\gamma\in (0,2)$.
Consider the following lift of $\mcal{S}_{\gamma}$:
\begin{equation*}
	\mcal{S}^{\mrm{univ}}:=\bigcup_{\substack{D\subsetneq \mbb{C}:\\ \mbox{\tiny 1-conn.}}}C_{\nabla}^{\infty}(D)^{\prime},
\end{equation*}
where $D\subsetneq \mbb{C}$ runs over all simply connected domains.
We consider the canonical surjection $\mcal{S}^{\mrm{univ}}\twoheadrightarrow \mcal{S}_{\gamma}$.
Notice that each component $C_{\nabla}^{\infty}(D)^{\prime}$ carries a right action
$a_{D,\gamma}:C_{\nabla}^{\infty}(D)^{\prime}\times \mrm{Aut}(D)\to C_{\nabla}^{\infty}(D)^{\prime}$ of $\mrm{Aut}(D)$
depending on the parameter $\gamma$ defined by
\begin{equation*}
	a_{D,\gamma}(h, \psi):=h\circ \psi +Q\log |\pr{\psi}|,\quad h\in C_{\nabla}^{\infty}(D),\quad \psi\in\mrm{Aut}(D),
\end{equation*}
where we set $Q=\frac{2}{\gamma}+\frac{\gamma}{2}$.
The chain rule ensures that $a_{D, \gamma}$ defines a right action of the group.
We write the quotient space $C_{\nabla}^{\infty}(D)^{\prime}/a_{D, \gamma}$ as $C_{\nabla}^{\infty}(D)^{\prime}_{\gamma-\mrm{red}}$
(we read ``$\gamma$-red" as ``$\gamma$-reduced"), and set
\begin{equation*}
	\widetilde{\mcal{S}_{\gamma}}:=\bigcup_{\substack{D\subsetneq \mbb{C}:\\ \mbox{\tiny 1-conn.}}}C_{\nabla}^{\infty}(D)^{\prime}_{\gamma-\mrm{red}}.
\end{equation*}

Let us introduce a groupoid $\mcal{G}=(\mcal{G}_{0},\mcal{G}_{1})$ whose objects are simply connected proper subdomains in $\mbb{C}$:
$\mcal{G}_{0}=\left\{D\subsetneq \mbb{C}:\mbox{simply connected}\right\}$
and the set of morphisms of which is given by
$\mcal{G}_{1}(D_{1},D_{2})=\mrm{Aut}(D_{2})\backslash \mrm{Iso}(D_{1},D_{2})/\mrm{Aut}(D_{1})$, $D_{1},D_{2}\in \mcal{G}_{0}$.
It is obvious that each set $\mcal{G}_{1}(D_{1},D_{2})$ consists of a single element, which we denote by the symbol $(D_{1}\to D_{2})$.
Then the anti-action of $\mcal{G}$ on $\widetilde{\mcal{S}_{\gamma}}$ is given as follows.
We consider a mapping
$c_{\gamma}:\mcal{G}_{1}\times \widetilde{\mcal{S}_{\gamma}}\to \widetilde{\mcal{S}_{\gamma}}$
which is defined for pairs $((D_{1}\to D_{2}),h\in C^{\infty}_{\nabla}(D_{2})^{\prime}_{\gamma-\mrm{red}})$, $D_{1},D_{2}\in\mcal{G}_{0}$ as
\begin{equation*}
	c_{\gamma}((D_{1}\to D_{2}),h):=h\circ \psi+Q\log|\pr{\psi}|\in C^{\infty}_{\nabla}(D_{1})^{\prime}_{\gamma-\mrm{red}},\ \ h\in C^{\infty}_{\nabla}(D_{2})^{\prime}_{\gamma-\mrm{red}},
\end{equation*}
where $Q=\frac{2}{\gamma}+\frac{\gamma}{2}$ and $\psi:D_{1}\to D_{2}$ is a conformal equivalence.
It can be verified that the above definition does not depend on the choice of a conformal equivalence $\psi$.
Then the quotient $\widetilde{\mcal{S}_{\gamma}}/c_{\gamma}$ is just the collection of $\gamma$-pre-quantum surfaces $\mcal{S}_{\gamma}$.

Consequently, the canonical quotient map $\mcal{S}^{\mrm{univ}}\to\mcal{S}_{\gamma}$ is the composition
\begin{equation*}
	\mcal{S}^{\mrm{univ}}\xrightarrow{/\bigcup_{D}a_{D,\gamma}} \widetilde{\mcal{S}_{\gamma}}\xrightarrow{/c_{\gamma}}\mcal{S}_{\gamma}.
\end{equation*}
By uniformizing any domain $D$ to the upper half plane, the collection $\mcal{S}_{\gamma}$ of
all $\gamma$-pre-quantum surfaces is identified with the space of $\gamma$-reduced distributions
$C^{\infty}_{\nabla}(\mbb{H})^{\prime}_{\gamma-\mrm{red}}$ on $\mbb{H}$.

\subsection{With marked boundary points}
Let us move on to $\mcal{S}_{\gamma ,N+1}$, $\gamma\in (0,2)$, $N\in\mbb{Z}_{\ge 0}$.
We consider the following space
\begin{equation*}
	\mcal{S}^{\mrm{univ}}_{N+1}=\bigcup_{\substack{D\subsetneq \mbb{C}:\\ 1-\mrm{conn.}}}C_{\nabla}^{\infty}(D)^{\prime}\times \mrm{Conf}_{N+1}^{<}(\del D),
\end{equation*}
where $D\subsetneq \mbb{C}$ runs over all simply connected domains.
As we have seen, for each $D\subsetneq \mbb{C}$, the component $C_{\nabla}^{\infty}(D)^{\prime}$ has a right action $a_{\gamma,D}$ of $\mrm{Aut}(D)$
depending on the parameter $\gamma$.
The same group also acts on $\mrm{Conf}_{N+1}^{<}(\del D)$ from the left.
We write the diagonal action of $\mrm{Aut}(D)$ on $C_{\nabla}^{\infty}(D)^{\prime}\times \mrm{Conf}^{<}_{N+1}(\del D)$ as $a_{D,\gamma,N+1}$
and set
\begin{equation*}
	\widetilde{\mcal{S}_{\gamma,N+1}}:=\bigcup_{\substack{D\subsetneq \mbb{C}:\\ 1-\mrm{conn.}}}C_{\nabla}^{\infty}(D)^{\prime}\times_{a_{D,\gamma,N+1}} \mrm{Conf}^{<}_{N+1}(\del D).
\end{equation*}
It can be verified that the groupoid $\mcal{G}$ again acts on $\widetilde{\mcal{S}_{\gamma,N+1}}$.
Then the space $\mcal{S}_{\gamma, N+1}$ is constructed as
\begin{equation*}
	\mcal{S}^{\mrm{univ}}_{N+1}\xrightarrow{/\bigcup_{D}a_{D,\gamma.N+1}} \widetilde{\mcal{S}_{\gamma,N+1}}\xrightarrow{/\mcal{G}}\mcal{S}_{\gamma,N+1}.
\end{equation*}

Let us write each component of $\widetilde{\mcal{S}_{\gamma,N+1}}$ as
\begin{equation*}
	\widetilde{\mcal{S}_{\gamma,N+1}}(D):=C_{\nabla}^{\infty}(D)^{\prime}\times_{a_{D,\gamma,N+1}} \mrm{Conf}^{<}_{N+1}(\del D).
\end{equation*}
Because the action of the groupoid $\mcal{G}$ is simply transitive,
the space $\mcal{S}_{\gamma, N+1}$ is noncanonically isomorphic to $\widetilde{\mcal{S}_{\gamma,N+1}}(D)$ for every $D$.

For simplicity, let us identify $\mcal{S}_{\gamma,N+1}$ with $\widetilde{\mcal{S}_{\gamma,N+1}}(\mbb{H})$
and consider the following commutative diagram:
\begin{equation*}
\xymatrix{
	C^{\infty}_{\nabla}(\mbb{H})^{\prime}\times \mrm{Conf}^{<}_{N+1}(\del\mbb{H}) \ar@{->>}[r] \ar@{->>}[d]& \mrm{Conf}^{<}_{N+1}(\del\mbb{H}) \ar@{->>}[d]\\
	\widetilde{\mcal{S}_{\gamma,N+1}}(\mbb{H}) \ar@{->>}[r]^{\pi} & \mrm{Aut}(\mbb{H})\backslash\mrm{Conf}^{<}_{N+1}(\del \mbb{H}).
}
\end{equation*}
\begin{enumerate}
\item 	If $N\le 2$, the space $\mrm{Aut}(\mbb{H})\backslash\mrm{Conf}^{<}_{N+1}(\del\mbb{H})=\{\ast\}$ consists of a single element.
		Thus the space $\mcal{S}_{\gamma,N+1}$ is isomorphic to the fiber $\pi^{-1}(\ast)$.
		\begin{enumerate}
		\item 	If $N=0$, the point $\infty \in\mrm{Conf}^{<}_{1}(\del\mbb{H})$ is fixed by the subgroup 
				$\mrm{Aff}(\mbb{H})=\{z\mapsto az+b|a>0, b\in\mbb{R}\}$
				of affine transformations in $\mrm{Aut}(\mbb{H})$.
				Thus, the fiber over $[\infty]\in \mrm{Aut}(\mbb{H})\backslash\mrm{Conf}^{<}_{1}(\del\mbb{H})$ becomes
				\begin{equation*}
					\pi^{-1}[\infty]\simeq C^{\infty}_{\nabla}(\mbb{H})^{\prime}/a_{\mbb{H},\gamma}(\mrm{Aff}(\mbb{H})).
				\end{equation*}
		\item 	If $N=1$, the point $(0,\infty)\in \mrm{Conf}^{<}_{2}(\del\mbb{H})$ is fixed by the subgroup
				$\mrm{Scl}(\mbb{H})=\{z\mapsto az|a>0\}$ of scale transformations in $\mrm{Aut}(\mbb{H})$.
				Thus, the fiber over $[0,\infty]\in\mrm{Aut}(\mbb{H})\backslash\mrm{Conf}^{<}_{2}(\del\mbb{H})$ becomes
				\begin{equation*}
					\pi^{-1}[0,\infty]\simeq C^{\infty}_{\nabla}(\mbb{H})^{\prime}/a_{\mbb{H},\gamma}(\mrm{Scl}(\mbb{H})).
				\end{equation*}
		\item 	If $N=2$, the action of $\mrm{Aut}(\mbb{H})$ on $\mrm{Conf}^{<}_{3}(\del\mbb{H})$ is simply transitive.
				Thus, the fiber over $[0,1,\infty]\in\mrm{Aut}(\mbb{H})\backslash\mrm{Conf}^{<}_{3}(\del\mbb{H})$ becomes
				\begin{equation*}
					\pi^{-1}[0,1,\infty]\simeq C^{\infty}_{\nabla}(\mbb{H})^{\prime}.
				\end{equation*}
		\end{enumerate}
\item 	If $N\ge 3$, the space $\mrm{Aut}(\mbb{H})\backslash\mrm{Conf}^{<}_{N+1}(\del\mbb{H})$ is $(N-2)$-dimensional over $\mbb{R}$
		and each fiber becomes
		\begin{equation*}
			\pi^{-1}[x_{1},\dots, x_{N+1}]\simeq C^{\infty}_{\nabla}(\mbb{H})^{\prime},\ \ (x_{1},\dots, x_{N+1})\in\mrm{Conf}^{<}_{N+1}(\del \mbb{H}).
		\end{equation*}
\end{enumerate}

We show an alternative construction of $\widetilde{\mcal{S}^{\mrm{Rot}}_{\gamma,N+1}}(\mbb{H})$ used to define
a $\gamma$-QS-$(N+1)$-MBPs of the $(\bm{X},\alpha)$-standard type in Sect.\ \ref{sect:formulation_problem}.
Let $\mrm{Rot}(\mbb{H})\subset \mrm{Aut}(\mbb{H})$ be the subgroup consisting of rotations of $\mbb{H}$.
We consider the following object:
\begin{equation*}
	\widetilde{\mcal{S}^{\mrm{Rot}}_{\gamma,N+1}}(\mbb{H}):=C^{\infty}_{\nabla}(\mbb{H})^{\prime}\times_{a_{\mbb{H},\gamma,N+1}(\mrm{Rot}(\mbb{H}))} \mrm{Conf}^{<}_{N+1}(\del\mbb{H}),
\end{equation*}
which is a fiber bundle over $\mrm{Rot}(\mbb{H})\backslash \mrm{Conf}^{<}_{N+1}(\del\mbb{H})$.
By sending the $(N+1)$-st point to $\infty$, we have
\begin{equation*}
	\mrm{Rot}(\mbb{H})\backslash \mrm{Conf}^{<}_{N+1}(\del\mbb{H})\simeq \mrm{Conf}^{<}_{N}(\mbb{R}),
\end{equation*}
and the fiber over $\bm{x}=(x_{1},\dots, x_{N})\in\mrm{Conf}^{<}_{N}(\mbb{R})$ is isomorphic to $C^{\infty}_{\nabla}(\mbb{H})^{\prime}$.
Since $\mrm{Conf}^{<}_{N}(\mbb{R})$ is contractible, the fiber bundle
$\widetilde{\mcal{S}^{\mrm{Rot}}_{\gamma,N+1}}(\mbb{H})\to\mrm{Conf}^{<}_{N}(\mbb{R})$ is trivial:
\begin{equation*}
	\widetilde{\mcal{S}^{\mrm{Rot}}_{\gamma,N+1}}(\mbb{H})\simeq C^{\infty}_{\nabla}(\mbb{H})^{\prime}\times \mrm{Conf}^{<}_{N}(\mbb{R}),
\end{equation*}
reducing to Eq.\ (\ref{eq:def_rot_space}).
In this construction, it becomes clear that the surjection $\pi^{\infty}_{\gamma,N+1}$ defined in (\ref{eq:surj_rot_to_surf}) is just the quotient map.

\section{Driving processes of a multiple SLE from an auxiliary function}
\label{app:driving_processes}
A time change of the Dyson model also appeared in ~\cite{BauerBernardKytola2005} as a particular example of a set of driving processes.
In their work, in connection to CFT, a set of driving processes $\bm{X}_{t}=(X^{(1)}_{t},\dots,X^{(N)}_{t})$, $t\ge 0$, $N\in\mbb{Z}_{\ge 1}$,
was derived from an auxiliary function $Z(x_{1},\dots, x_{N})$ annihilated by operators
\begin{equation*}
	\mcal{D}_{i}=\frac{\kappa}{2}\del_{x_{i}}^{2}+2\sum_{j;j\neq i}\left(\frac{1}{x_{i}-x_{j}}\del_{x_{j}}+\frac{h_{\kappa}}{(x_{i}-x_{j})^{2}}\right),\ \ i=1,\dots,N,
\end{equation*}
where $h_{\kappa}=\frac{\kappa-6}{2\kappa}$ so that $\{\bm{X}_{t}=(X^{(1)}_{t},\dots, X^{(N)}_{t}):t\ge 0\}$ satisfies the system of SDEs
\begin{equation*}
	dX^{(i)}_{t}=\sqrt{\kappa}dB^{(i)}_{t}+\kappa(\del_{x_{i}}\log Z)(\bm{X}_{t})dt+\sum_{j;j\neq i}\frac{2dt}{X^{(i)}_{t}-X^{(j)}_{t}},\ \ t\ge 0,\ \ i=1,\dots, N.
\end{equation*}
Our set of driving processes (\ref{eq:driving_process_Schramm}) associated with functions (\ref{eq:function_canonical})
can be obtained by taking the following auxiliary function:
\begin{equation*}
	Z(x_{1},\dots, x_{N})=\prod_{i<j}(x_{i}-x_{j})^{\frac{2}{\kappa}}.
\end{equation*}
This auxiliary function is a correlation function of the Coulomb gas
and was argued in ~\cite{BauerBernardKytola2005} to be related to interfaces according to the criterion by Dub{\'e}dat \cite{Dubedat2007}.

It is also possible to directly derive the reverse flow of a multiple SLE in an analogous way as the one used in ~\cite{BauerBernardKytola2005}.
The central idea is to require correlation functions of a CFT evaluated by certain stochastic processes to be local martingales.
Let us begin with the reverse flow of a single SLE \cite{Lawler2009b,ViklundLawler2012}:
\begin{equation*}
	\frac{d}{dt} g^{\mrm{R}}_{t}(z)=-\frac{2}{g^{\mrm{R}}_{t}(z)-\sqrt{\kappa}B_{t}},\ \ t\ge 0,\ \ \ g^{\mrm{R}}_{0}(z)=z\in\mbb{H},
\end{equation*}
where $\{B_{t}:t\ge 0\}$ is a standard Brownian motion and $\kappa>0$ is a parameter.
If we define $f^{\mrm{R}}_{t}(z):=g^{\mrm{R}}_{t}(z)-\sqrt{\kappa}B_{t}$, then it satisfies
\begin{equation*}
	df^{\mrm{R}}_{t}(z)=-\frac{2dt}{f^{\mrm{R}}_{t}(z)}-\sqrt{\kappa}dB_{t},\ \ t\ge 0,\ \ z\in\mbb{H}.
\end{equation*}
Now let us recall the group theoretical formulation of SLE \cite{BauerBernard2003, BauerBernard2004} (see also ~\cite[Sect.\ II]{SK2018}),
in which this evolution can be enhanced to an operator-valued stochastic process acting on the representation space
of the Virasoro algebra.
The resulting stochastic process denoted as $R(f_{t}^{\mrm{R}})$ satisfies
\begin{equation*}
	R(f_{t}^{\mrm{R}})^{-1}dR(f_{t}^{\mrm{R}})=\left(2L_{-2}+\frac{\kappa}{2}L_{-1}^{2}\right)dt + \sqrt{\kappa}L_{-1}dB_{t},\ \ t\ge 0,\ \ \ R(f_{0})=\mrm{Id},
\end{equation*}
where $L_{n}$, $n\in\mbb{Z}$ are the standard Virasoro generators.
By the standard argument (see also ~\cite{Fukusumi2017}), for a highest weight vector $\ket{c,h}$ for the Virasoro algebra,
the stochastic process $R(f_{t}^{\mrm{R}})\ket{c,h}$, $t\ge 0$, is a local martingale if the highest weight is chosen as
\begin{equation*}
	c=c^{\mrm{R}}_{\kappa}=1+\frac{3(\kappa+4)^{2}}{2\kappa},\ \ \ h=h_{\kappa}^{\mrm{R}}=-\frac{\kappa+6}{2\kappa}.
\end{equation*}
Note that the same local martingale is also expressed as
\begin{equation*}
R(f_{t}^{\mrm{R}})\ket{c_{\kappa}^{\mrm{R}},h_{\kappa}^{\mrm{R}}}=R(g_{t}^{\mrm{R}})\Psi_{h_{\kappa}^{\mrm{R}}}(\sqrt{\kappa}B_{t})\ket{0},
\end{equation*}
where $\Psi_{h_{\kappa}^{\mrm{R}}}$ is the primary field of conformal weight $h_{\kappa}^{\mrm{R}}$.

The multiple analogue of the reverse flow of an SLE would require a CFT datum as an input.
Let us consider the following correlation function:
\begin{equation*}
	Z(x_{1},\cdots,x_{N})=\braket{h_{\infty}|\Psi_{h_{\kappa}^{\mrm{R}}}(x_{1})\cdots \Psi_{h_{\kappa}^{\mrm{R}}}(x_{N})|0},
\end{equation*}
the central charge for which is $c_{\kappa}^{\mrm{R}}$.
It follows from the existence of a singular vector $(2L_{-2}+\frac{\kappa}{2}L_{-1}^{2})\ket{c_{\kappa}^{\mrm{R}},h_{\kappa}^{\mrm{R}}}$ that
the correlation function $Z(x_{1},\cdots,x_{N})$ solves the following system of differential equations:
\begin{equation}
	\label{eq:annihilating_condition}
	\mcal{D}_{i}^{\mrm{R}}Z=0,\ \ \  i=1,\dots, N,
\end{equation}
where 
\begin{equation*}
	\mcal{D}_{i}^{\mrm{R}}=\frac{\kappa}{2}\del_{x_{i}}^{2}-2\sum_{j;j\neq i}\left(\frac{1}{x_{j}-x_{i}}\del_{x_{j}}-\frac{h_{\kappa}^{\mrm{R}}}{(x_{j}-x_{i})^{2}}\right),\ \ \ i=1,\cdots, N.
\end{equation*}

\begin{exam}
The following function gives a solution to the system of differential equations (\ref{eq:annihilating_condition}):
\begin{equation}
\label{eq:auxiliary_reverse}
	Z(x_{1},\cdots,x_{N})=\prod_{i<j}(x_{i}-x_{j})^{-\frac{2}{\kappa}}.
\end{equation}
\end{exam}

As in the case of the forward flow of a multiple SLE in ~\cite{BauerBernardKytola2005}, we define driving processes as follows:
\begin{defn}
Let $\{B_{t}^{(i)}:t\ge 0\}_{i=1}^{N}$ be mutually independent standard Brownian motions,
and $Z(x_{1},\cdots,x_{N})$ be a solution to the system of differential equations (\ref{eq:annihilating_condition}).
The associated set of driving processes $\{\bm{Y}^{(i)}_{t}=(Y_{t}^{(1)},\dots, Y^{(N)}_{t}):t\ge 0\}$ is defined as a solution to the system of SDEs:
\begin{equation*}
	dY_{t}^{(i)}=\sqrt{\kappa}dB_{t}^{(i)}+\kappa (\del_{x_{i}} \log Z)(\bm{Y}_{t})dt-\sum_{j;j\neq i}\frac{2dt}{Y_{t}^{(i)}-Y_{t}^{(j)}},\ \ i=1,\dots, N,\ \ t\ge 0.
\end{equation*}
\end{defn}

\begin{exam}
For the function given in (\ref{eq:auxiliary_reverse}), the associated set of driving processes satisfies (\ref{eq:driving_process_reverse}).
Thus the reverse flow of multiple SLE considered in Sect.\ \ref{sect:coupling_reverse} is a particular example of ones defined below.
\end{exam}

Associated with these data, we define the reverse flow of a multiple SLE as follows:
\begin{defn}
Let $Z(x_{1},\cdots,x_{N})$ be a solution to the system of differential equations (\ref{eq:annihilating_condition})
and $\{\bm{Y}_{t}:t\ge 0\}$ be the associated driving processes.
The reverse flow of an multiple SLE$_{\kappa}$ associated with these data is a stochastic process $\{g^{\mrm{R}}_{t}(\cdot)\}_{t\ge 0}$ solving
the following multiple version of the reverse flow:
\begin{equation*}
	\frac{d}{dt}g^{\mrm{R}}_{t}(z)=-\sum_{i=1}^{N}\frac{2}{g^{\mrm{R}}_{t}(z)-Y_{t}^{(i)}},\ \ t\ge 0,\ \ \ g^{\mrm{R}}_{0}(z)=z.
\end{equation*}
\end{defn}

The reverse flow defined above is connected to CFT in the following sense:
\begin{thm}
Let $Z(x_{1},\cdots,x_{N})$ be a solution to the system of differential equations (\ref{eq:annihilating_condition}),
$\{\bm{Y}_{t}:t\ge 0\}$ be the associated set of driving processes,
and let $\{g_{t}^{\mrm{R}}(\cdot)\}_{t\ge 0}$ be the corresponding reverse flow of an multiple SLE.
Then the stochastic process
\begin{equation*}
	M_{t}:=\frac{1}{Z(\bm{Y}_{t})}R(g_{t}^{\mrm{R}})\Psi_{h_{\kappa}^{\mrm{R}}}(Y_{t}^{(1)})\cdots \Psi_{h_{\kappa}^{\mrm{R}}}(Y_{t}^{(N)})\ket{0},\ \ t\ge 0,
\end{equation*}
is a local martingale on a representation space of the Virasoro algebra.
\end{thm}

\bibliographystyle{alpha}
\bibliography{sle_gff}
\end{document}